\documentclass[11pt]{article}
\usepackage[a4paper, total={161mm, 241mm}]{geometry}
\usepackage{graphicx}
\usepackage{amsmath}
\usepackage{amssymb}
\usepackage{epstopdf}
\usepackage{bm}
\usepackage{hyperref}
\usepackage{amsthm}
\usepackage{todonotes}
\usepackage{minipage}
\usepackage{algpseudocode}
\usepackage{caption}
\usepackage{newfloat}
\usepackage[square,numbers,sort]{natbib}
\usepackage{soul}
\renewcommand{\ul}[1]{\underline{\bm{#1}}}

\DeclareFloatingEnvironment[fileext=frm]{algfloat}
\usepackage{tikz}
\usetikzlibrary{math}
\usetikzlibrary{shapes.geometric, decorations.markings, arrows}
\tikzstyle{box} = [rectangle, rounded corners, minimum width=3cm, minimum height=1cm,text centered, draw=black]
\tikzstyle{arrow} = [thick,->,>=stealth']

\let\oldnl\nl
\newcommand{\nonl}{\renewcommand{\nl}{\let\nl\oldnl}}

\theoremstyle{plain}
\newtheorem{theorem}{Theorem}[section]
\newtheorem{definition}[theorem]{Definition}

\newtheorem{remark}[theorem]{Remark}

\theoremstyle{definition}

\newcommand{\intd}[1]{\text{d}#1}

\title{Waveform Relaxation with asynchronous time-integration}
\author{Peter Meisrimel$^{\mbox{\tiny\rm *}, \mbox{\tiny\rm 1}}$, Philipp Birken$^{\mbox{\tiny\rm 1}}$}
\begin{document}
\maketitle
\baselineskip=0.9
\normalbaselineskip
\vspace{-3pt}
\begin{center}{\footnotesize\em $^{\mbox{\tiny\rm 1}}$Centre for the
    mathematical sciences, Numerical Analysis, Lund University, Lund, Sweden\\ email: peter.meisrimel\symbol{'100}na.lu.se, philipp.birken\symbol{'100}na.lu.se, $^{\mbox{\tiny\rm *}}$Corresponding author}
\end{center}

\begin{abstract}
We consider Waveform Relaxation (WR) methods for partitioned time-integration of surface-coupled multiphysics problems. WR allows independent time-discretizations on independent and adaptive time-grids, while maintaining high time-integration orders. Classical WR methods such as Jacobi or Gauss-Seidel WR are typically either parallel or converge quickly. 

We present a novel parallel WR method utilizing asynchronous communication techniques to get both properties. Classical WR methods exchange discrete functions after time-integration of a subproblem. We instead asynchronously exchange time-point solutions during time-integration and directly incorporate all new information in the interpolants. We show both continuous and time-discrete convergence in a framework that generalizes existing linear WR convergence theory. An algorithm for choosing optimal relaxation in our new WR method is presented.

Convergence is demonstrated in two conjugate heat transfer examples. Our new method shows an improved performance over classical WR methods. In one example we show a partitioned coupling of the compressible Euler equations with a nonlinear heat equation, with subproblems implemented using the open source libraries \texttt{DUNE} and \texttt{FEniCS}.
\end{abstract}
\medskip\\ 
{The authors gratefully acknowledge support from the Swedish e-science collaboration eSSENCE.}
%
%
%
\section{Introduction}
%
We consider \textit{multiphysics problems}, which are comprised of coupled systems with different physics. In particular, we consider problems with a \textit{bidirectional surface coupling. I.e., the subproblems interact} via a lower dimensional interface. Examples are within \textit{fluid structure interaction} in the simulation of blood flow in large arteries \cite{Crosetto2011}, cooling of rocket engines \cite{Kowollik2013a,Kowollik2013} or gas quenching \cite{Yarrington1994}.

We follow the \textit{partitioned approach}, which allows re-use of existing codes and solving the subproblems with different computational methods on individual grids. Our focus is time-integration. We want to solve subproblems using independent and higher-order time-discretizations on adaptive time-grids. Additionally, we want to perform time-integration of the subproblems in parallel, on top of a parallelization in space.

A technique that promises to meet all these requirements is the so called Waveform relaxation (WR). An iteration requires solving the subproblems on a time window. Thereby, continuous interface functions, obtained via suitable interpolation, are provided from the respective other problem. 
WR methods were originally introduced in \cite{Lelarasmee1982} for systems of ordinary differential equations (ODEs), and used for the first time to solve time dependent PDEs in \cite{Gander1998,Giladi2002}. WR appears in the literature under a variety of names: Waveform relaxation/iteration, dynamic iteration/relaxation and Picard(-Lindel\"{o}f) iteration.

The most common type of WR methods are Gauss-Seidel (GS) WR, solving all subproblems in sequence, and Jacobi WR, which solves all subproblems in parallel. However, the parallelism of Jacobi WR typically comes at the cost of slower convergence rates, due to less information exchange.

In this article, we construct a novel and inherently parallel WR method with more information exchange than Jacobi WR. Our ansatz to increase communication is to exchange the results of each timestep directly after computation. Any new information is directly incorporated by updating the interpolants, affecting their subsequent evaluations. This increases the information exchange and thus enhances convergence rates. We use asynchronous \textit{One-sided-communication} that allows solving subproblems in parallel and does not require function calls on the receiving processor.

WR methods require convergence acceleration to achieve fast convergence rates. We consider classical convergence acceleration by weighting updates using relaxation parameters, which are highly problem specific \cite{Janssen1997,Miekkala1987}. Other acceleration techniques involve using an additional convolution relaxation term \cite{Reichelt1995,Janssen1997} or Krylov-subspace acceleration \cite{Lumsdaine2003}, see \cite{Lumsdaine2003} for a wider overview of different acceleration techniques. However, many of these are not applicable in the partitioned approach. Black-box convergence acceleration techniques such as quasi-Newton methods can also be applied to WR and have been shown to work well \cite{Ruth2020}. 

With classical WR methods, data dependencies between the subproblems are fixed. In our new method, dependencies can vary in time and differ in each iteration, since One-sided-communication is not deterministic. We present an analytical description of our new method and convergence proofs in the continuous and time-discrete setting for linear problems. This generalizes existing WR theory \mbox{\cite{Janssen1997}}.

We present an algorithm for optimal relaxation in our new method. Here, optimal relaxation critically depends on the realized communication, which is not deterministic. Thus, we deduce the realized communication between the subsolvers in every timestep and choose suitable relaxation for each time-point solution.

We demonstrate our method using two conjugate heat transfer test cases, showing convergence in both. In the first test case, two coupled heterogeneous linear heat equations, performance results show a runtime speed-up of our new method compared to classical Jacobi and GS WR methods. The second experiment is a gas quenching test case, which consists of the compressible Euler equations coupled to a nonlinear heat equation. Here, we demonstrate a black-box coupling of heterogeneous space discretizations and subsolver codes. The fluid is solved using a finite volume discretization implemented in \texttt{DUNE} \mbox{\cite{Bastian2021}} and the solid is solved via a finite element discretization implemented in \texttt{FEniCS} \mbox{\cite{logg2012_FENICS}}.

The paper is structured as follows: We first introduce general continuous and time-discrete WR methods in Sections \ref{SEC WR CONT} and \ref{SEC WR DISCR}. In Section \ref{SEC RMA} we provide a brief overview over the principles of one-sided asynchronous communication. We formalize our new approach in Section \ref{SEC NEW} and present a first algorithm. We discuss convergence in the linear case in Section \ref{SEC AWR CONV}, showing time-discrete and continuous convergence. Our algorithm for choosing optimal relaxation for two coupled problems is shown in Section \ref{SEC AWR RELAX}. 
Finally, we show numerical results, followed by summary and conclusions.
%
\section{Continuous Waveform Relaxation}\label{SEC WR CONT}
%
Consider the following coupled system of initial value problems
\begin{align}\label{EQ WR BASE NONLIN}
\begin{split}
\dot{\bm{v}}(t) & = \bm{g}(t, \bm{v}(t), \bm{w}(t)), \qquad \bm{v}(0) = \bm{v}_0 \in \mathbb{R}^{d_v},\\
\dot{\bm{w}}(t) & = \bm{h}(t, \bm{v}(t), \bm{w}(t)), \qquad \bm{w}(0) = \bm{w}_0 \in \mathbb{R}^{d_w},
\end{split}
\quad t \in [0, T_f < \infty].
\end{align}
We now define a general continuous Waveform Relaxation method. Given $\bm{v}^{(k)}$ and $\bm{w}^{(k)}$, a single iteration consists of solving two differential equations and performing two relaxation steps as follows:
\begin{subequations}\label{EQ WR ITER}
\begin{align}
\dot{\hat{\bm{v}}}^{(k+1)}(t) 
& = \bm{g}\left(t, \hat{\bm{v}}^{(k+1)}(t), \bm{w}_{\,*}^{(k)}(t)\right), \quad \hat{\bm{v}}^{(k+1)}(0) = \bm{v}_0,
\quad t \in [0, T_f], \label{EQ WR ITER 1}\\
\bm{v}^{(k+1)}
& = \bm{v}^{(k)} + \boldsymbol{\Theta}_v (\hat{\bm{v}}^{(k+1)} -\bm{v}^{(k)}), \label{EQ WR RELAX 1}\\
\dot{\hat{\bm{w}}}^{(k+1)}(t) 
& = \bm{h}\left(t, \bm{v}_{\,*}^{(k)}(t), \hat{\bm{w}}^{(k+1)}(t)\right), \quad \hat{\bm{w}}^{(k+1)}(0) = \bm{w}_0,
\quad t \in [0, T_f], \label{EQ WR ITER 2} \\
\bm{w}^{(k+1)}
& = \bm{w}^{(k)} + \boldsymbol{\Theta}_w(\hat{\bm{w}}^{(k+1)} - \bm{w}^{(k)}), \label{EQ WR RELAX 2}
\end{align}
\end{subequations}
with nonsingular diagonal matrices $\boldsymbol{\Theta}_v \in \mathbb{R}^{d_v \times d_v}$, $\boldsymbol{\Theta}_w \in \mathbb{R}^{d_w \times d_w}$ for relaxation. Extensions to more than two systems are straight-forward, c.f. \cite{Vandewalle2013}. The specific WR method is defined by the choices for $\bm{v}_{\,*}^{(k)}$ and $\bm{w}_{\,*}^{(k)}$. The trivial initial guesses for $\bm{v}^{(0)}$ and $\bm{w}^{(0)}$ are to extrapolate the initial value.

The most common WR methods are \textit{Gauss-Seidel} (GS) and \textit{Jacobi} WR, c.f., \cite{White1987a,Vandewalle2013}. Continuous GS WR is given by 
\begin{equation}\label{EQ GS WR}
\bm{v}_{\,*}^{(k)} = \bm{v}^{(k+1)}
\quad \text{and} \quad
\bm{w}_{\,*}^{(k)} = \bm{w}^{(k)}.
\end{equation}
GS WR is sequential, which makes it sensitive to the order of the systems in \eqref{EQ WR BASE NONLIN}.

Jacobi WR is given by 
\begin{equation}\label{EQ JAC WR}
\bm{v}_{\,*}^{(k)} = \bm{v}^{(k)}
\quad \text{and} \quad
\bm{w}_{\,*}^{(k)} = \bm{w}^{(k)},
\end{equation}
which allows parallel computation of \eqref{EQ WR ITER 1}, \eqref{EQ WR RELAX 1} and \eqref{EQ WR ITER 2}, \eqref{EQ WR RELAX 2}.

The iteration is commonly terminated if
\begin{equation}\label{EQ WR TERMINATION CRIT}
\| \bm{y}^{(k+1)}(T_f) - \bm{y}^{(k)}(T_f) \| < \| \bm{y}^{(k+1)}(0)\|\, TOL_{WR},
\end{equation}
i.e., via the relative update measured at $t = T_f$, where updates tend to be the largest. Here, $\bm{y}$ is a subset of the unknowns of $\bm{v}$ resp. $\bm{w}$ that $\bm{h}$ resp. $\bm{g}$ in \eqref{EQ WR BASE NONLIN} depends on.
%
\section{Time-discrete Waveform Relaxation}\label{SEC WR DISCR}
%
We enable the use of independent time-grids and time-integration schemes by using interpolants of the respective discrete solutions in the right-hand sides of \mbox{\eqref{EQ WR ITER 1}} and \mbox{\eqref{EQ WR ITER 2}}.

We denote discrete solutions by 
\begin{equation*}
\ul{\bm{v}}^{(k)} := \{ \bm{v}^{(k)}_n\}_{n = 0, \ldots, N_v^{(k)}}
\quad \text{and} \quad
\ul{\bm{w}}^{(k)} := \{ \bm{w}^{(k)}_n\}_{n = 0, \ldots, N_w^{(k)}},
\end{equation*}
on time-grids $0 = t_0^{(v),(k)} < \ldots < t_{N_v^{(k)}}^{(v),(k)} = T_f$ and $0 = t_0^{(w),(k)} < \ldots < t_{N_w^{(k)}}^{(w),(k)} = T_f$. The interpolants are as follows: 
\begin{equation*}
\mathcal{I}(\ul{\bm{v}}^{(k)}) \in \mathcal{C}\left([0, T_f]; \mathbb{R}^{d_v}\right)
\quad \text{and} \quad
\mathcal{I}(\ul{\bm{w}}^{(k)}) \in \mathcal{C}\left([0, T_f]; \mathbb{R}^{d_w}\right).
\end{equation*}
Here, we omit the time-grids as input to the interpolants for ease of notation. We obtain a time-discrete WR method by using discrete time-integration to solve
\begin{subequations}\label{EQ WR DISCRETE}
\begin{align}
\dot{\hat{\bm{v}}}^{(k+1)}(t) 
& = \bm{g}\left(t, \hat{\bm{v}}^{(k+1)}(t), \mathcal{I}(\ul{\bm{w}}_{\,*}^{(k)})(t)\right),
\quad \hat{\bm{v}}^{(k+1)}(0) = \bm{v}_0,
\quad t \in [0, T_f], \label{EQ WR DISCRETE 1}\\
\dot{\hat{\bm{w}}}^{(k+1)}(t) 
& = \bm{h}\left(t, \mathcal{I}(\ul{\bm{v}}_{\,*}^{(k)})(t), \hat{\bm{w}}^{(k+1)}(t)\right),
\quad \hat{\bm{w}}^{(k+1)}(0) = \bm{w}_0,
\quad t \in [0, T_f]. \label{EQ WR DISCRETE 2}
\end{align}
\end{subequations}
Here, one chooses $\ul{\bm{v}}_{\,*}^{(k)}$, $\ul{\bm{w}}_{\,*}^{(k)}$ in accordance with e.g., \eqref{EQ GS WR} or \eqref{EQ JAC WR}. We consider \textit{polynomial interpolation}. 

Relaxation is performed in the discrete data-points as follows
\begin{subequations}
\begin{align}
\bm{v}^{(k+1)}_n &
= \left(\bm{I} - \boldsymbol{\Theta_v} \right) \mathcal{I}(\ul{\bm{v}}^{(k)})(t_n^{(v),(k+1)})
+ \boldsymbol{\Theta_v}\hat{\bm{v}}^{(k+1)}_n
,
\quad n = 1, \ldots, N_v^{(k+1)},
\label{EQ WR RELAX DISCR V}
\\
\bm{w}^{(k+1)}_n &
= \left(\bm{I} - \boldsymbol{\Theta_w} \right) \mathcal{I}(\ul{\bm{w}}^{(k)})(t_n^{(w),(k+1)})
+ \boldsymbol{\Theta_w}\hat{\bm{w}}^{(k+1)}_n
,
\quad n = 1, \ldots, N_w^{(k+1)}.
\label{EQ WR RELAX DISCR W}
\end{align}
\end{subequations}
See Algorithm \ref{ALG WR DISCR JAC} for a pseudocode of time-discrete Jacobi WR. We use interpolation with evaluation at run-time. I.e., we define the interpolants once using fixed data-structures for $\ul{\bm{v}}_{\,*}$ and $\ul{\bm{w}}_{\,*}$. We then update $\ul{\bm{v}}_{\,*}$ and $\ul{\bm{w}}_{\,*}$ during the iteration, affecting the results of subsequent evaluations. Our discrete initial guesses for $\bm{v}^{(0)}$ and $\bm{w}^{(0)}$ are to extrapolate the initial value.

\begin{algfloat}[h!]
\small
\begin{center}
\Statex \textbf{Pseudocode: Time-discrete Jacobi WR}
\end{center}
\begin{tabular}{c|c}
\begin{minipage}{.48\textwidth}
\begin{algorithmic}[1]
\Statex \textbf{Process} $0$ (\textbf{p}$0$)
\State $\ul{\bm{w}}^{(0)}$ discrete initial guess
\State Initialize $\ul{\bm{w}}_{\,*}$ and $\mathcal{I}(\ul{\bm{w}}_{\,*})$
\For{$k = 0, \ldots, k_{\max}-1$}
\State $\ul{\bm{w}}_{\,*} \gets \ul{\bm{w}}^{(k)}$ Update interpolant
\State $\ul{\hat{\bm{v}}}^{(k+1)} \gets$ Discr. solve \eqref{EQ WR DISCRETE 1}
\State $\ul{\bm{v}}^{(k+1)} \gets$ Relaxation \eqref{EQ WR RELAX DISCR V}
\State $\ul{\bm{v}}^{(k+1)} \rightarrow$ Send to \textbf{p}$1$
\State $\ul{\bm{w}}^{(k+1)} \gets$ Recv. from \textbf{p}$1$
\State Check \eqref{EQ WR TERMINATION CRIT}, \texttt{break} if \texttt{true}
\EndFor
\end{algorithmic}
\end{minipage}
&
\begin{minipage}{.48\textwidth}
\begin{algorithmic}[1,nolinenumbers]
\Statex \textbf{Process} $1$ (\textbf{p}$0$)
\State $\ul{\bm{v}}^{(0)}$ discrete initial guess
\State Initialize $\ul{\bm{v}}_{\,*}$ and $\mathcal{I}(\ul{\bm{v}}_{\,*})$
\For{$k = 0, \ldots, k_{\max}-1$}
\State $\ul{\bm{v}}_{\,*} \gets \ul{\bm{v}}^{(k)}$ Update interpolant
\State $\ul{\hat{\bm{w}}}^{(k+1)} \gets$ Discr. solve \eqref{EQ WR DISCRETE 2}
\State $\ul{\bm{w}}^{(k+1)} \gets$ Relaxation \eqref{EQ WR RELAX DISCR W}
\State $\ul{\bm{w}}^{(k+1)} \rightarrow$ Send to \textbf{p}$1$
\State $\ul{\bm{v}}^{(k+1)} \gets$ Recv. from \textbf{p}$1$
\State Check \eqref{EQ WR TERMINATION CRIT}, \texttt{break} if \texttt{true}
\EndFor
\end{algorithmic}
\end{minipage}\\\\\hline
\end{tabular}
\caption{Pseudocode of the time-discrete Jacobi WR method. Obtaining the initial guesses may involve communication. Here, $\ul{\bm{v}}_{\,*}^{(k)}$ and $\ul{\bm{w}}_{\,*}^{(k)}$ in \eqref{EQ WR DISCRETE} are defined by \eqref{EQ JAC WR}.}
\label{ALG WR DISCR JAC}
\end{algfloat}
%
\section{One-sided communication}\label{SEC RMA}
%
The standard in parallel computations is Point-to-Point communication, primarily using \texttt{MPI\_Send} and \texttt{MPI\_Recv}. Here, every \texttt{MPI\_Send} requires a matching \texttt{MPI\_Recv} function call.
This works well for algorithms with fixed synchronization points, e.g., the termination check in Algorithm \ref{ALG WR DISCR JAC}. Here, we want to perform time-integration of the subproblems in parallel on independent grids, exchanging information after each timestep.

We solve this by using One-sided communication, also called remote memory access (RMA) \cite[Chpt.11]{MessagePassingInterfaceForum2012}, which is asynchronous. Since it is non-standard, we give a brief overview over the associated concepts and methods. It is also worth noting that RMA via \texttt{MPI} is not available in \texttt{mpi4py}, requiring an implementation in \texttt{C++} or \texttt{Fortran}.

\texttt{MPI\_Window} objects and their allocated memory facilitate RMA. One accesses the memory associated with a given window via \texttt{MPI\_Get} (read) or \texttt{MPI\_Put} (write) operations. The memory of a target window is only accessible during access periods. We use \textit{passive target synchronization}, in which a processor creates an access period on a target window (including windows on its own memory) by locking and unlocking the target window using \texttt{MPI\_Win\_(un)lock}. This does not require active participation in terms of function calls on the processor of the targeted window. An additional parameter in the \texttt{MPI\_Win\_(un)lock} call can specify the type of lock to ensure exclusive access for writing or shared access for reading, preventing access conflicts. The order in which locks are obtained is determined at runtime and can differ with each execution, which means RMA is not deterministic.

The RMA memory model uses public and private copies of windows. All \texttt{MPI\_Put} and \texttt{MPI\_Get} operations act on the public copy. However, the present variable values are given by the private window copies. As such, a process needs to synchronize its public and private window copies to obtain any updates received via \texttt{MPI\_Put} operations. This is done explicitly using \texttt{MPI\_Win\_sync}. Some RMA functions perform this synchronization implicitly, note that \texttt{MPI\_Barrier} does not. 
%
\section{Waveform relaxation with asynchronous time-integration}\label{SEC NEW}
%
While Jacobi WR is parallel, convergence rates are typically slower than those of GS WR, due to less information exchange. The goal is to develop a parallel WR method with more information exchange than Jacobi WR and thus a faster convergence rate. Our ansatz is to use asynchronous communication already during time-integration to increase the information exchange.

We start with the parallel Algorithm \ref{ALG WR DISCR JAC} and modify it to increase communication. First, we expose the interpolant data $\ul{\bm{v}}_{\,*}$ and $\ul{\bm{w}}_{\,*}$ via \texttt{MPI\_Window} objects. Thus, remote updates are directly incorporated in subsequent interpolant evaluations. Next, we move relaxation and communication to the time-step level, i.e., into the time-stepping loop. We asynchronously communicate new time-point solutions, remotely updating the corresponding values in $\ul{\bm{v}}_{\,*}$ resp. $\ul{\bm{w}}_{\,*}$ on the other process, using \texttt{MPI\_Put}.

Algorithm \ref{ALG WR NEW BASE} shows the pseudocode for two coupled problems, with a different number of timesteps $N_v \neq N_w$ for each subproblem and with constant relaxation. In Section \ref{SEC AWR RELAX} we present an algorithm with variable relaxation, based on the realized communication.

\begin{algfloat}[h!]
\small
\begin{center}
\Statex \textbf{Pseudocode: WR with asynchronous time-integration}
\end{center}
\begin{tabular}{c|c}
\begin{minipage}{.46\textwidth}
\begin{algorithmic}[1]
\Statex \textbf{Process} $0$ (\textbf{p}$0$)
\State $\ul{\bm{w}}^{(0)}$ discrete initial guess
\State Initialize $\ul{\bm{w}}_{\,*} \gets \ul{\bm{w}}^{(0)}$ and $\mathcal{I}(\ul{\bm{w}}_{\,*})$ \label{LINE INTERPOLANT}
\State Expose $\ul{\bm{w}}_{\,*}$ via \texttt{MPI\_Window}
\For{$k = 0, \ldots, k_{\max}-1$}
\For{$n = 1, \ldots, N_v$}
\State $\hat{\bm{v}}^{(k+1)}_n \gets$ Solve* \eqref{EQ WR DISCRETE 1}
\State $\bm{v}^{(k+1)}_n \gets$ Relaxation \eqref{EQ WR RELAX DISCR V}
\State $\bm{v}^{(k+1)}_n \rightarrow$ \texttt{MPI\_Put} to $\ul{\bm{v}}_{\,*}$ on \textbf{p}$1$
\EndFor
\State Sync. + Termination check
\EndFor
\end{algorithmic}
\end{minipage}
&
\begin{minipage}{.46\textwidth}
\begin{algorithmic}[1,nolinenumbers]
\State \textbf{Process} $1$ (\textbf{p}$1$)
\State $\ul{\bm{v}}^{(0)}$ discrete initial guess
\State Initialize $\ul{\bm{v}}_{\,*} \gets \ul{\bm{v}}^{(0)}$ and $\mathcal{I}(\ul{\bm{v}}_{\,*})$
\State Expose $\ul{\bm{v}}_{\,*}$ via \texttt{MPI\_Window}
\For{$k = 0, \ldots, k_{\max}-1$}
\For{$n = 1, \ldots, N_w$}
\State $\hat{\bm{w}}^{(k+1)}_n \gets$ Solve* \eqref{EQ WR DISCRETE 2}
\State $\bm{w}^{(k+1)}_n \gets$ Relaxation \eqref{EQ WR RELAX DISCR W}
\State $\bm{w}^{(k+1)}_n \rightarrow$ \texttt{MPI\_Put} to $\ul{\bm{w}}_{\,*}$ on \textbf{p}$0$
\EndFor
\State Sync. + Termination check
\EndFor
\end{algorithmic}
\end{minipage}\\\\\hline
\end{tabular}
\caption{New proposed method using asynchronous communication during time-integration. Obtaining the initial guesses may involve communication. Solve* denotes a single discrete timestep in solving \eqref{EQ WR DISCRETE 1} resp. \eqref{EQ WR DISCRETE 2}, with the interpolants defined in Line \ref{LINE INTERPOLANT} in the right-hand sides.}
\label{ALG WR NEW BASE}
\end{algfloat}

Our new method is defined by 
\begin{equation}\label{EQ AWR VSTAR}
(\ul{\bm{v}}_{\,*}^{(k)})_n = 
\begin{cases}
\bm{v}_{n}^{(k+1)}, & \text{if available} \\
\bm{v}_{n}^{(k)}, & \text{else}
\end{cases},
\end{equation}
$\ul{\bm{w}}_{\,*}^{(k)}$ analogous. Availability is determined at run-time, by the present data when evaluating the interpolant. Since asynchronous communication is not deterministic, $\ul{\bm{v}}_{\,*}^{(k)}$ and $\ul{\bm{w}}_{\,*}^{(k)}$ can vary for different timesteps and with $k$. Due to remote updates, evaluations of the interpolants for the same $t$ and $k$, but at different real-life times can differ.

The corresponding continuous WR method is \eqref{EQ WR ITER}, with $\bm{v}_{\,*}^{(k)}$ and $\bm{w}_{\,*}^{(k)}$ varying in both time and with $k$.
%
\subsection{Variable relaxation}\label{SEC AWR VAR RELAX}
%
With our new method, $\bm{v}_{\,*}^{(k)}$ and $\bm{w}_{\,*}^{(k)}$ vary with $t$ and $k$. It is possible that we obtain Jacobi or GS WR. Since optimal relaxation matrices can notably differ for Jacobi and GS WR, constant relaxation is unlike to achieve optimal convergence acceleration in our new method. We instead consider relaxation varying with $t$ and $k$:
\begin{align}\label{EQ VAR RELAX}
\begin{split}
\bm{v}^{(k+1)}(t)
& = \bm{v}^{(k)}(t) + \boldsymbol{\Theta}_v^{(k+1)}(t) \left(\hat{\bm{v}}^{(k+1)}(t) -\bm{v}^{(k)}(t)\right),\\
\bm{w}^{(k+1)}(t)
& = \bm{w}^{(k)}(t) + \boldsymbol{\Theta}_w^{(k+1)}(t)\left(\hat{\bm{w}}^{(k+1)}(t) - \bm{w}^{(k)}(t)\right),
\end{split}
\end{align}
where $\boldsymbol{\Theta}_v^{(k+1)}(t)$ and $\boldsymbol{\Theta}_w^{(k+1)}(t)$ are non-singular diagonal matrices. We first discuss convergence results, since these provide us the necessary insight on how to choose relaxation. Following that, we extend Algorithm \ref{ALG WR NEW BASE} to variable relaxation in Section \ref{SEC AWR RELAX}. There, we present a concrete approach to determining $\bm{v}_{\,*}^{(k)}$ and $\bm{w}_{\,*}^{(k)}$ at runtime, to then choose appropriate relaxation for each timestep.
%
\section{Convergence Analysis}\label{SEC AWR CONV}
%
Similar to \cite{Nevanlinna1989,Nevanlinna1989a,Janssen1996a,Janssen1996} we analyze convergence in the linear setting. We first present the established WR theory, before extending it to include our new method. Consider the following \textit{monolithic system}
\begin{equation}\label{EQ WR BASE LIN GENERAL}
\bm{B} \dot{\bm{u}}(t) + \bm{A} \bm{u}(t) = \bm{f}(t)
, \quad
\bm{u}(0) = \bm{u}_0 \in \mathbb{R}^{d}
, \quad
t \in [0, T_f < \infty]
, \quad 
\bm{B} \text{ nonsingular}
,
\end{equation}
with $\bm{B}$, $\bm{A}$ $\in \mathbb{R}^{d \times d}$, and $\bm{f}$ Lipschitz-continuous.
Here, one can express classical WR methods such as Jacobi and GS WR via constant splittings \cite{Nevanlinna1989,Janssen1996a}
\begin{equation}\label{EQ WR SPLIT CONST}
\bm{B} = \bm{M}_B - \bm{N}_B
\quad \text{and} \quad
\bm{A} = \bm{M}_A - \bm{N}_A,
\quad 
\bm{M}_B \text{ nonsingular},
\end{equation}
and the iteration
\begin{align}\label{EQ WR ITER CONST}
\begin{split}
\bm{M}_B \dot{\bm{u}}^{(k+1)}(t) + \bm{M}_A \bm{u}^{(k+1)}(t)
= \bm{N}_B \dot{\bm{u}}^{(k)}(t) + \bm{N}_A \bm{u}^{(k)}(t) + \bm{f}(t),
\\
\bm{u}^{(k+1)}(0) = \bm{u}_0, 
\quad t \in [0, T_f].
\end{split}
\end{align}
The particular splitting \mbox{\eqref{EQ WR SPLIT CONST}} depends on the WR method, e.g., Jacobi or GS WR, and includes constant relaxation \mbox{\eqref{EQ WR RELAX 1}}, \mbox{\eqref{EQ WR RELAX 2}}. We omit dependencies of the splitting matrices on the relaxation matrices in \mbox{\eqref{EQ WR SPLIT CONST}} for readability.

Consider for example the following system of ODEs:
\begin{equation*}
\underbrace{\begin{pmatrix}
\bm{B}_1 & \bm{B}_2 \\ \bm{B}_3 & \bm{B}_4
\end{pmatrix}}_{\bm{B}}
\underbrace{\begin{pmatrix}
\dot{\bm{v}}(t) \\ \dot{\bm{w}}(t)
\end{pmatrix}}_{\dot{\bm{u}}(t)}
+ 
\underbrace{\begin{pmatrix}
\bm{A}_1 & \bm{A}_2 \\ \bm{A}_3 & \bm{A}_4
\end{pmatrix}}_{\bm{A}}
\underbrace{\begin{pmatrix}
\bm{v}(t) \\ \bm{w}(t)
\end{pmatrix}}_{\bm{u}(t)}
=
\underbrace{
\begin{pmatrix}
\bm{f}_1(t) \\ \bm{f}_2(t)
\end{pmatrix}}_{\bm{f}(t)},
\quad
\underbrace{
\begin{pmatrix}
\bm{v}(0) \\ \bm{w}(0)
\end{pmatrix}}_{\bm{u}(0)}
=
\underbrace{
\begin{pmatrix}
\bm{v}_0 \\ \bm{w}_0
\end{pmatrix}}_{\bm{u}_0}
,
\end{equation*}
with $t \in [0, T_f]$ and $\bm{B}_1$, $\bm{B}_4$ nonsingular. Then, Jacobi WR, without relaxation, is given by
\begin{align*}
\begin{split}
&
\begin{pmatrix}
\bm{B}_1 & \bm{0} \\
\bm{0} & \bm{B}_4
\end{pmatrix}
\dot{\bm{u}}^{(k+1)}(t)
+
\begin{pmatrix}
\bm{A}_1 & \bm{0} \\
\bm{0} & \bm{A}_4
\end{pmatrix}
\bm{u}^{(k+1)}(t)
\\
= &
\begin{pmatrix}
\bm{0} & -\bm{B}_2\\
-\bm{B}_3 & \bm{0}
\end{pmatrix}
\dot{\bm{u}}^{(k)}(t)
+
\begin{pmatrix}
\bm{0} & -\bm{A}_2\\
-\bm{A}_3 & \bm{0}
\end{pmatrix}
\bm{u}^{(k)}(t)
+ 
\bm{f}(t),
\end{split}
\end{align*}
with $t \in [0, T_f]$ and $\bm{u}^{(k+1)}(0) = \bm{u}_0$. The inherent parallelism of this method is reflected by the block-diagonal structure of the matrices on the left-hand side.

One determines the convergence properties of continuous WR methods by analyzing the iteration \mbox{\eqref{EQ WR ITER CONST}}. Similarly, time-discrete WR methods are described via time-discretizations of \mbox{\eqref{EQ WR ITER CONST}}.
%
\subsection{Time-discrete WR with asynchronous communication}
%
We now consider the time-discrete case. Similar to \cite{Janssen1996,Nevanlinna1989a}, we use convergent and zero-stable $m$-step linear multistep methods (LMM), see e.g. \cite[Chap.3.2]{haiwan:93}, on matching time-grids with the constant step-size $\Delta t = T_f/N$. A $m$-step LMM applied to \eqref{EQ WR BASE LIN GENERAL} is
\begin{equation}\label{EQ LMM}
\sum_{\ell=0}^m \left( a_\ell \bm{B} + b_\ell \Delta t \bm{A} \right)\bm{u}_{n+\ell} 
= \Delta t \sum_{\ell=0}^m b_\ell\, \bm{f}(t_{n + \ell}),
\quad n = m \ldots, N,
\end{equation}
which defines the \textit{discrete monolithic solution}. A particular LMM is defined by its coefficients $a_\ell, b_\ell$ and requires starting values $\bm{u}_\ell$, $\ell = 0, \ldots, m-1$.

Classical WR methods such as GS and Jacobi WR can be described by time-discretizations of \mbox{\eqref{EQ WR ITER CONST}}. Using the same LMM, this is
\begin{align*}
\begin{split}
\sum_{\ell=0}^m \left(a_\ell \bm{M}_{B} + b_\ell \Delta t\, \bm{M}_{A}\right) \bm{u}_{n+\ell}^{(k+1)}
=
\sum_{\ell=0}^m \left(a_\ell\bm{N}_{B} + b_\ell \Delta t\, \bm{N}_{A}\right)\bm{u}^{(k)}_{n+\ell}
 + b_\ell \Delta t\, \bm{f}(t_{n+\ell}), 
\quad n = m \ldots, N,
\end{split}
\end{align*}
with starting values $\bm{u}_\ell^{(k)} = \bm{u}_\ell$, $\ell = 0, \ldots, m-1$, $\forall k \geq 0$.

We can describe Algorithm \mbox{\ref{ALG WR NEW BASE}} by an analogous iteration, where the splitting matrices \mbox{\eqref{EQ WR SPLIT CONST}} can differ for each $\ell$, $n$ and $k$. That is, we consider
\begin{align}\label{EQ AWR DISCRETE ITER LMM}
\begin{split}
&\sum_{\ell=0}^m \left(a_\ell \bm{M}_{B,n,\ell}^{(k+1)} + b_\ell \Delta t\, \bm{M}^{(k+1)}_{A, n, \ell}\right) \bm{u}_{n+\ell}^{(k+1)}\\
= &
\sum_{\ell=0}^m \left(a_\ell\bm{N}_{B,n,\ell}^{(k+1)} + b_\ell \Delta t\, \bm{N}^{(k+1)}_{A, n, \ell}\right)\bm{u}^{(k)}_{n+\ell}
 + b_\ell \Delta t\, \bm{f}(t_{n+\ell}),
\quad n = m \ldots, N.
\end{split}
\end{align}
Here, the concrete matrices $\bm{M}_{B,n,\ell}^{(k+1)}$ and $\bm{M}_{A,n,\ell}^{(k+1)}$ are, for each $\ell$, $n$ and $k$, determined by $\ul{\bm{v}}_*^{(k)}$ and $\ul{\bm{w}}_*^{(k)}$, as emerging from the realized communication in e.g., Algorithm \mbox{\ref{ALG WR NEW BASE}}, including relaxation. These matrices fulfill the splitting property $\bm{B} = \bm{M}_{B,n,\ell}^{(k+1)} - \bm{N}_{B,n,\ell}^{(k+1)}$, $\bm{A}$ analogous, by which the discrete monolithic solution defined by \mbox{\eqref{EQ LMM}} is a fixed point of the discrete WR method defined by \mbox{\eqref{EQ AWR DISCRETE ITER LMM}}.

We define the discrete WR error as
\begin{equation*}
\bm{e}_n^{(k)} := \bm{u}_n - \bm{u}_n^{(k)}.
\end{equation*}
Taking the difference between \eqref{EQ LMM} and \eqref{EQ AWR DISCRETE ITER LMM} shows that it fulfills
\begin{equation}\label{EQ DISCRETE ITER LMM ERROR}
\sum_{\ell=0}^m \bm{C}^{(k+1)}_{n, \ell} \bm{e}_{n+\ell}^{(k+1)}
=
\sum_{\ell=0}^m \bm{D}^{(k+1)}_{n, \ell} \bm{e}^{(k)}_{n+\ell},
\end{equation}
with
\begin{equation*}
\bm{C}^{(k)}_{n, \ell} 
:= a_\ell \bm{M}_{B,n,\ell}^{(k)} + b_\ell\, \Delta t\bm{M}^{(k)}_{A, n, \ell}
,\quad 
\bm{D}^{(k)}_{n, \ell}
:= 
a_\ell\bm{N}_{B,n,\ell}^{(k)} + b_\ell\,\Delta t \bm{N}^{(k)}_{A, n, \ell}.
\end{equation*}
The starting values $\bm{u}^{(k+1)}_\ell$ define the starting errors $\bm{e}^{(k+1)}_{\ell}$, $\ell = 0, \ldots, m-1$.

In the following theorem we show convergence in the form of $\| \bm{e}^{(k)}_{n} \| \rightarrow 0$, for $k \rightarrow \infty$, for all $n = m, \ldots, N$. It is an extension of the convergence result from \mbox{\cite{Janssen1996}}. There, we have constant splittings, whereas in our method the splittings vary with $n$, $\ell$ and $k$. Here $\| \cdot \| : \mathbb{R}^d \rightarrow \mathbb{R}$ is a norm and we similar use $\| \cdot \|$ to denote the induced matrix norm.
\begin{theorem}\label{THRM AWR DISCRETE CONV}
Let the splittings
\begin{equation*}
\bm{B} = \bm{M}_{B,n,\ell}^{(k)} - \bm{N}_{B, n, \ell}^{(k)}
\quad \text{and} \quad
\bm{A} = \bm{M}_{A,n,\ell}^{(k)} - \bm{N}_{A, n, \ell}^{(k)}
, \quad 
\bm{M}_{B, n, \ell}^{(k)}
\text{ nonsingular},
\end{equation*}
with
\begin{equation}\label{EQ AWR DISCRETE THRM CONDS}
\bm{C}_{n,m}^{(k)} \text{ nonsingular}
,\quad
\| {\bm{C}_{n,m}^{(k)}}^{-1} \bm{D}_{n,m}^{(k)}\| < 1,
\quad n = m, \ldots, N
\end{equation}
and an initial guess 
\begin{equation*}
\ul{e}^{(0)} 
:= 
\left( {\bm{e}_m^{(0)}}^T, \ldots, {\bm{e}_N^{(0)}}^T \right)^T
\in \mathbb{R}^{d(N-m + 1)}
\end{equation*}
be given. Then, the solution of the discrete WR method defined by \mbox{\eqref{EQ AWR DISCRETE ITER LMM}} converges to the solution of \mbox{\eqref{EQ LMM}}.

%
\end{theorem}
\begin{proof}
We consider the so called "all-at-once system" system, which is the system for all timesteps of a given iteration. This is
\begin{equation*}
\ul{C}^{(k+1)} \ul{e}^{(k+1)} = \ul{D}^{(k+1)} \ul{e}^{(k)},
\end{equation*}
with 
\begin{align*}
\ul{C}^{(k)} & := 
\begin{pmatrix}
\bm{C}^{(k)}_{m, m} & \bm{0} & \ldots& \ldots & \ldots& \bm{0}\\
\vdots & \bm{C}^{(k)}_{m+1, m} &\ddots & & & \vdots\\
\bm{C}^{(k)}_{m^2, 0} & & \ddots & \ddots & & \vdots\\
\bm{0} & \ddots & & \ddots & \ddots & \vdots\\
\vdots & \ddots & \ddots & & \ddots & \bm{0}\\
\bm{0} & \ldots & \bm{0} & \bm{C}^{(k)}_{N, 0} & \ldots & \bm{C}^{(k)}_{N, m}
\end{pmatrix}
, 
\\
\ul{D}^{(k)} & := 
\begin{pmatrix}
\bm{D}^{(k)}_{m, m} & \bm{0} & \ldots& \ldots & \ldots& \bm{0}\\
\vdots & \bm{D}^{(k)}_{m+1, m} &\ddots & & & \vdots\\
\bm{D}^{(k)}_{m^2, 0} & & \ddots & \ddots & & \vdots\\
\bm{0} & \ddots & & \ddots & \ddots & \vdots\\
\vdots & \ddots & \ddots & & \ddots & \bm{0}\\
\bm{0} & \ldots & \bm{0} & \bm{D}^{(k)}_{N, 0} & \ldots & \bm{D}^{(k)}_{N, m}
\end{pmatrix},
\end{align*}
with $\ul{C}$, $\ul{D} \in \mathbb{R}^{(d(N-m + 1))\times (d(N-m + 1))}$.
%
Using standard index notation for $\ul{C}^{(k)}$ to reference the above blocks, its inverse is
\begin{align*}
{\ul{C}^{(k)}}^{-1} = 
\begin{pmatrix}
\bm{G}^{(k)}_{1, 1} & \bm{0} & \ldots& \bm{0}\\
\vdots & \ddots &\ddots & \vdots \\
\vdots & & \ddots & \bm{0}\\
\bm{G}^{(k)}_{N-m+1, 1} & \ldots & \ldots & \bm{G}^{(k)}_{N-m+1, N-m+1}
\end{pmatrix},
\\
\bm{G}^{(k)}_{i, j} = 
\begin{cases}
\left(\ul{C}_{i, j}^{(k)}\right)^{-1} ,& \quad i = j \\
- \left(\ul{C}_{i, i}^{(k)}\right)^{-1} \sum_{\ell = 1}^{i - 1} \ul{C}^{(k)}_{i, \ell} \bm{G}_{\ell, j},& \quad i \neq j
\end{cases}.
\end{align*}
The inverse ${\ul{C}^{(k)}}^{-1}$ only requires inverses of its diagonal blocks $\bm{C}^{(k)}_{n,m}$, for which we assume existence, see \eqref{EQ AWR DISCRETE THRM CONDS}. Since the resulting iteration matrix ${\ul{C}^{(k)}}^{-1} \ul{D}^{(k)}$ is block lower-triangular, we get the following forward elimination:
\begin{align}\label{EQ DISCRETE ERR RECURRENCE}
\begin{split}
\bm{e}_m^{(k+1)}
& = {\bm{C}_{m,m}^{(k+1)}}^{-1} \bm{D}_{m,m}^{(k+1)} \bm{e}_m^{(k)}
,\\
\bm{e}_{m+1}^{(k+1)}
& = {\bm{C}_{m+1,m}^{(k+1)}}^{-1} \bm{D}_{m+1,m}^{(k+1)} \bm{e}_{m+1}^{(k)}
\\
& \quad + \left(\bm{G}_{2,1}^{(k+1)} \bm{D}_{m,m}^{(k+1)} + {\bm{C}_{m+1,m}^{(k+1)}}^{-1} \bm{D}_{m+1,m-1}^{(k+1)}\right) \bm{e}_m^{(k)}
,\\
\bm{e}_{m+2}^{(k+1)}
& = {\bm{C}_{m+2,m}^{(k+1)}}^{-1} \bm{D}_{m+2,m}^{(k+1)} \bm{e}_{m+2}^{(k)}
\\
& \quad + \left( \bm{G}_{3,2}^{(k+1)} \bm{D}_{m+1,m-1}^{(k+1)} + {\bm{C}_{m+2,m}^{(k+1)}}^{-1} \bm{D}_{m+2,m-1}^{(k+1)}\right) \bm{e}_{m+1}^{(k)}
\\ & \quad + \left( \bm{G}_{3,1}^{(k+1)} \bm{D}_{m,m}^{(k+1)} + \bm{G}_{3,2}^{(k+1)} \bm{D}_{m+1,m-1}^{(k+1)} + {\bm{C}_{m+2,m}^{(k+1)}}^{-1} \bm{D}_{m+2,m-1}^{(k+1)}\right) \bm{e}_m^{(k)}
,\\
\bm{e}_{m+3}^{(k+1)} & = \ldots
.
\end{split}
\end{align}
We get $\| \bm{e}_m^{(k+1)} \| = 0$, $k \rightarrow \infty$ from $\|{\bm{C}_{m,m}^{(k+1)}}^{-1} \bm{D}_{m,m}^{(k+1)}\| < 1$. With $\bm{e}_m^{(k+1)}$ vanishing and $\|{\bm{C}_{m+1,m}^{(k+1)}}^{-1} \bm{D}_{m+1,m}^{(k+1)}\| < 1$, we get $\| \bm{e}_{m+1}^{(k+1)} \| = 0$, $k \rightarrow \infty$. By induction we get $\| \bm{e}_{n}^{(k+1)} \| = 0$, $k \rightarrow \infty$, for all $n = m, \ldots, N$, which implies $\| \ul{e}^{(k+1)}\| = 0$, $k \rightarrow \infty$.
\end{proof}
The assumption of $\bm{C}_{n,m}^{(k)}$ nonsingular in \eqref{EQ AWR DISCRETE THRM CONDS} is a solvability assumption on the occurring linear systems in \eqref{EQ DISCRETE ITER LMM ERROR}. The growing number of terms on the right hand sides in \eqref{EQ DISCRETE ERR RECURRENCE} shows the potential for initial growth of $\| \ul{e}^{(k)} \|$, even if the iteration matrix may be normal. 

\begin{remark}\label{REM THREE OPTIONS}
Consider Algorithm \mbox{\ref{ALG WR NEW BASE}} with matching time-grids, i.e., $N = N_v = N_w$ and constant relaxation. Then, the splitting matrices ${\bm{M}_{B,n,\ell}^{(k)}}$ resp. $\bm{M}_{A,n,\ell}^{(k)}$ match either those of Jacobi or GS WR for each $\ell, n, k$, due to matching $(\ul{\bm{v}}_*^{(k)})_n$ and $(\ul{\bm{w}}_*^{(k)})_n$, for each $n$, see \eqref{EQ AWR VSTAR}. Thus, one only needs to consider three distinct cases for the matrices \mbox{\eqref{EQ AWR DISCRETE THRM CONDS}}. Each one of these is a convergence requirement for either Jacobi or GS WR. Consequently, time-discrete convergence of Jacobi and GS WR (in either ordering of \mbox{\eqref{EQ WR BASE NONLIN}}) means \mbox{\eqref{EQ AWR DISCRETE THRM CONDS}} is met.
\end{remark}
%
\subsection{Continuous WR with asynchronous communication}
%
We now consider continuous WR methods with $\bm{w}_{\,*}^{(k)}$ and $\bm{v}_{\,*}^{(k)}$ in \eqref{EQ WR ITER 1}, \eqref{EQ WR ITER 2} varying with $t$ and $k$. By straight-forward substitutions of $\hat{\bm{v}}^{(k+1)}$ and $\hat{\bm{w}}^{(k+1)}$ in the variable relaxation steps \mbox{\eqref{EQ VAR RELAX}} into \mbox{\eqref{EQ WR ITER}}, we get the iteration
\begin{align}\label{EQ WR SPLIT VAR ITER}
\begin{split}
&\bm{M}_B^{(k+1)}(t) \dot{\bm{u}}^{(k+1)}(t) + \bm{M}_A^{(k+1)}(t) \bm{u}^{(k+1)}(t)\\
= & \,\,\bm{N}_B^{(k+1)}(t) \dot{\bm{u}}^{(k)}(t) + \bm{N}_A^{(k+1)}(t) \bm{u}^{(k)}(t) + \bm{f}(t)
,\quad \bm{u}^{(k+1)}(0) = \bm{u}_0
,\quad t \in [0, T_f],
\end{split}
\end{align}
with splittings
\begin{equation}\label{EQ WR SPLIT VAR}
\bm{B} = \bm{M}_B^{(k)}(t) - \bm{N}_B^{(k)}(t),
\quad 
\bm{A} = \bm{M}_A^{(k)}(t) - \bm{N}_A^{(k)}(t),
\quad 
\bm{M}_B^{(k)}(t) \text{ nonsingular},
\end{equation}
$t \in [0, T_f]$, $k > 0$. Here, we omit the dependencies on the relaxation matrices for readability.

In the previous section, we considered the convergence of the fully discrete WR iteration \eqref{EQ AWR DISCRETE ITER LMM} for $\Delta t$ fixed for $k \rightarrow \infty$, which gives \eqref{EQ LMM}. Here, we instead discuss the convergence of the continuous iteration \eqref{EQ WR SPLIT VAR ITER} for $k\rightarrow \infty$. Before doing so, we would like to point out that we cannot guarantee that we obtain \eqref{EQ WR SPLIT VAR ITER} from \eqref{EQ AWR DISCRETE ITER LMM} in the limit $\Delta t \rightarrow 0$. The reason is that the splittings chosen and thus the matrices $\bm{M}_{B,n,\ell}^{(k+1)}$ resp. $\bm{M}_{A,n,\ell}^{(k+1)}$ in \eqref{EQ AWR DISCRETE ITER LMM} can change with every time step. This would yield $\bm{M}_B^{(k)}$, $\bm{M}_A^{(k)}$ discontinuous everywhere in the limit and \eqref{EQ WR SPLIT VAR ITER} would not be well defined. 

The typical scenario for \eqref{EQ AWR DISCRETE ITER LMM}, as implemented via Algorithm \ref{ALG WR NEW BASE}, is that splittings match those of Jacobi WR until one subsolver is at least one timestep ahead of another subsolver. From then on, the splitting matrices match those of GS WR. I.e., for a given $k$, $\bm{M}_B^{(k)}$ and $\bm{M}_A^{(k)}$ are piece-wise constant, with a single discontinuity. 

We thus assume that the limit has only a finite number of jumps and consider \eqref{EQ WR SPLIT VAR ITER} in a piece-wise sense with piece-wise Lipschitz-continuous data. This guarantees existence of a piece-wise solution of \eqref{EQ WR SPLIT VAR ITER} for all $k > 0$. Additionally, we assume that splitting matrices corresponding to the same time-point, in the same iteration, are identical. E.g., in \eqref{EQ AWR DISCRETE ITER LMM} both $\bm{M}_{B,n,\ell}^{(k+1)}$ and $\bm{M}_{B,n+1,\ell-1}^{(k+1)}$ correspond to $t_{n+\ell}$. This can be guaranteed by implementation, storing interpolant evaluations. Now we analyze the convergence properties of \eqref{EQ WR SPLIT VAR ITER} under these assumptions.

Consider \eqref{EQ WR BASE LIN GENERAL} with $\bm{A}$, $\bm{B}$, $\bm{f}$ time-dependent and piece-wise Lipschitz-continuous. Then the solution is
\begin{equation}\label{EQ CONT SOL}
\bm{u}(t) = 
\text{e}^{-C(t)} 
\left(
\bm{u}_0
+
\int_0^t \text{e}^{C(s)} \bm{B}^{-1}(s) \bm{f}(s) \intd{s}
\right),
\end{equation}
where
\begin{equation*}
C(t) = \int_0^t \bm{B}^{-1}(s) \bm{A}(s) \intd{s}.
\end{equation*}
We can apply this solution formula to \eqref{EQ WR SPLIT VAR ITER}. Replacing $\dot{\bm{u}}^{(k)}$ via integration by parts and performing lengthy, but straight-forward rearrangements, yield the solution:
\begin{equation}\label{EQ WF ITER NEW SOL}
\bm{u}^{(k+1)}(t) = 
\bm{K}^{(k+1)}(t) \bm{u}^{(k)}(t) 
+ 
\int_0^t \boldsymbol{\mathcal{K}}_c^{(k+1)}(s) \bm{u}^{(k)}(s) \intd{s}
+
\boldsymbol{\varphi}^{(k+1)}(t),
\end{equation}
with
\begin{align}\label{EQ WF ITER NEW SOL KERNELS}
\begin{split}
\bm{K}^{(k)}(t) 
& = 
{\bm{M}_B^{(k)}}^{-1}(t)\bm{N}_B^{(k)}(t), 
\\
\bm{C}^{(k)}(t) 
&= 
\int_0^t {\bm{M}_B^{(k)}}^{-1}(s) \bm{M}_A^{(k)}(s) \intd{s},
\\
\boldsymbol{\mathcal{K}}_c^{(k)}(t) 
&=
\text{e}^{\bm{C}^{(k)}(s) - \bm{C}^{(k)}(t)}
\left( {\bm{M}_B^{(k)}}^{-1}(s) \bm{N}_A(s) 
- \frac{\text{d}}{\intd{s}} \left( \text{e}^{\bm{C}^{(k)}(s)} \right) \bm{K}^{(k)}(s)
- \frac{\text{d}}{\intd{s}} \bm{K}^{(k)}(s)
\right),
\\
\boldsymbol{\varphi}^{(k)}(t)
& =
\text{e}^{- \bm{C}^{(k)}(t)} 
\left(
\left( \bm{I} - \bm{K}^{(k)}(0) \right)\bm{u}_0
+ \int_0^t \text{e}^{\bm{C}^{(k)}(s)} {\bm{M}_B^{(k)}}^{-1}(s) \bm{f}(s) \intd{s}
\right), 
\end{split}
\end{align}
where
\begin{equation*}
\frac{\text{d}}{\intd{s}} \left( \text{e}^{\bm{C}^{(k)}(s)} \right)
=
\int_0^1 \text{e}^{\alpha \bm{C}^{(k)}(s)} \frac{\text{d} \bm{C}^{(k)}(s)}{\intd{s}} \text{e}^{(1 - \alpha) \bm{C}^{(k)}(s)} \intd{\alpha},
\end{equation*}
c.f. \cite{Wilcox1967}. 

Consider the continuous WR error
\begin{equation}\label{EQ AWR ERROR}
\bm{e}^{(k)} := \bm{u}^{(k)} - \bm{u},
\end{equation}
where $\bm{u}$ is the solution to \eqref{EQ WR BASE LIN GENERAL}. Taking the difference between \eqref{EQ WF ITER NEW SOL} and \eqref{EQ CONT SOL} gives:
\begin{equation*}
\bm{e}^{(k+1)}(t)
= \bm{K}^{(k+1)}(t) \bm{e}^{(k)}(t)
+ \int_0^t \boldsymbol{\mathcal{K}}_c^{(k+1)}(s) \bm{e}^{(k)}(s) \intd{s},
\quad 
\bm{e}^{(k+1)}(0) = \bm{0},
\quad t \in [0, T_f].
\end{equation*}
For the following theorem, we define the (vector) function norm $\| \bm{e} \|_{[0, t]} := \sup_{\tau \in [0, t]} \| \bm{e}(\tau)\|$ and (matrix) function norm $\| \bm{A} \|_{[0, t]} := \sup_{\tau \in [0, t]} \| \bm{A}(\tau)\|$, based on the induced matrix norm. This result is an extension of a convergence result in \mbox{\cite{Janssen1996a}}, to the situation where splittings vary with $t$ and $k$.
%
\begin{theorem}
Let splittings \eqref{EQ WR SPLIT VAR} with $\bm{M}_B^{(k)}$, $\bm{M}_A^{(k)}$ and $\bm{e}^{(0)}$ piece-wise Lipschitz-continuous for all $k > 0$ be given. Then, the error \eqref{EQ AWR ERROR} fulfills
%
\begin{equation*}
\| \bm{e}^{(k)} \|_{[0, t]}
\leq
\left( \sum_{j = 0}^k {k \choose j} {K^{\,\max}(t)}^{k-j} {\mathcal{K}_c^{\,\max}(t)}^{j} \frac{t^{\,j}}{j!} \right) \| \bm{e}^{(0)} \|_{[0, t]},
\end{equation*}
%
where $K^{\,\max}(t) := \sup_{k\in \mathbb{N}} \| \bm{K}^{(k)} \|_{[0, t]}$ and $\mathcal{K}_c^{\,\max}(t) := \sup_{k\in \mathbb{N}}\|\boldsymbol{\mathcal{K}}_c^{(k)}\|_{[0, t]}$, c.f. \eqref{EQ WF ITER NEW SOL KERNELS}.
\end{theorem}
\begin{proof}
To avoid ambiguity for the function norm, we denote the relevant variable by $\tau$.
Following the same principles as in the constant splitting case \cite{Kranenborg2018,Janssen1996a}, applications of the triangle inequality, submultiplicativity and straight-forward upper bounds yield
\begin{align*}
\| \bm{e}^{(k)} \|_{[0, t]}
& = 
\left\| \bm{K}^{(k)} \bm{e}^{(k-1)} + 
\int_0^{\tau} \boldsymbol{\mathcal{K}}_c^{(k)}(s) \bm{e}^{(k-1)}(s) \intd{s} \right\|_{[0, t]} \\
& \leq 
\big\| \bm{K}^{(k)} \bm{e}^{(k-1)} \big\|_{[0, t]}
+ \int_0^t \big\|\boldsymbol{\mathcal{K}}_c^{(k)}(s)\bm{e}^{(k-1)}(s)\big\| \intd{s} \\
%
%
%
& \leq 
\big\| \bm{K}^{(k)} \big\|_{[0, t]} \big\| \bm{e}^{(k-1)} \big\|_{[0, t]}
+ \big\| \boldsymbol{\mathcal{K}}_c^{(k)} \big\|_{[0, t]} \int_0^t \big\|\bm{e}^{(k-1)}\big\|_{[0, t]} \intd{s} 
.
\end{align*}
Repeated application and taking the supremum over $k$ then gives 
\begin{align*}
\| \bm{e}^{(k)} \|_{[0, t]}
& \leq \left( \sum_{j = 0}^k {k \choose j} K^{\,\max}(t)^{k-j} \mathcal{K}_c^{\,\max}(t)^{j} \int \ldots \int 1\,\, \intd{s}_1 \ldots \intd{s}_{j} \right) \| \bm{e}^{(0)} \|_{[0, t]} \\
& \leq
\left( \sum_{j = 0}^k {k \choose j} K^{\,\max}(t)^{k-j} \mathcal{K}_c^{\,\max}(t)^{\,j} \frac{t^{\,j}}{j\,!} \right) \| \bm{e}^{(0)}\|_{[0, t]}.
\end{align*}
\end{proof}
Our results is consistent with the time-discrete result of Theorem \mbox{\ref{THRM AWR DISCRETE CONV}} for $\Delta t \rightarrow 0$, under the aforementioned assumptions. The $\mathcal{K}_c^{\,\max}(t)^{\,j} \frac{t^{\,j}}{j\,!}$ term converges super-linearly for $j \rightarrow \infty$, but can lead to large error bounds for small $j$ and large $t$. The asymptotic convergence rate for $k \rightarrow \infty$ is bounded from above by $\| K^{\,\max}\|_{[0, t]}$.
%
\section{Variable relaxation algorithm for two coupled problems}\label{SEC AWR RELAX}
%
We now provide an algorithm and an implementation for variable relaxation when using asynchronous communication. Theorem \ref{THRM AWR DISCRETE CONV}, which includes variable relaxation, shows that the discrete asymptotic convergence rate is bounded by
\begin{align}\label{EQ AWR RELAX START}
\begin{split}
\max_{n = m, \ldots, N,\,\, k > 0}\left\|{\bm{C}_{n,m}^{(k)}}^{-1} \bm{D}_{n,m}^{(k)}\right\|.
\end{split}
\end{align}
These ${\bm{C}_{n,m}^{(k)}}^{-1} \bm{D}_{n,m}^{(k)}$ are the diagonal blocks of the iteration matrix and they depend on the chosen relaxation. Thus, we choose relaxation to minimize the spectral radii, resp. norms of all diagonal blocks, which minimizes \mbox{\eqref{EQ AWR RELAX START}}. The optimal relaxation depends on the problem and the splitting. We discuss how to determine the specific values for our numerical experiments in Section \mbox{\ref{SEC NUM RELAX}}.

Here, we present an algorithm for the separate processes to determine at runtime which splitting occurs in each timestep. With two coupled problems, there are exactly three cases, corresponding to Jacobi and GS WR, see Remark \mbox{\ref{REM THREE OPTIONS}}. This algorithm does not have a straight-forward extension to more than two coupled problems. We consider relaxation of the subset of unknowns exchanged between the processors. Furthermore, we consider the more general case of non-matching time-grids with constant stepsizes.

In the following we reference the subsolver for a given subproblem as process, not excluding usage of multiple processors to solve a subproblem.
The basic structure for the WR iteration and time-integration are analogous to Algorithm \mbox{\ref{ALG WR NEW BASE}}. The differences will be within interpolation, communication and relaxation.
\begin{definition}
In Algorithm \mbox{\ref{ALG WR NEW BASE}}, we say a process is {\bf ahead} of another process, if all interpolant evaluations in the $n$-th timestep and $k$-th iteration depend on data-points from the $k$-th iteration, rather than the $(k-1)$-st iteration.
\end{definition}
\begin{definition}
We say a timestep has a {\bf local Jacobi shape}, if no process is ahead of another process, and a {\bf local GS shape}, if one process is ahead of the other.
\end{definition}
Our new algorithm for variable relaxation consists of first determining the local shape and then updating the interpolant, using appropriate relaxation. Thus, we first communicate $\hat{\bm{v}}^{(k+1)}_n$ and $\hat{\bm{w}}^{(k+1)}_n$ to a buffer on the respective other process. Then, that process determines the local shape and updates the interpolant data $\ul{\bm{v}}_{\,*}$ resp. $\ul{\bm{w}}_{\,*}$ using appropriate relaxation. With relaxation independent of the local shapes, e.g., Algorithm \mbox{\ref{ALG WR NEW BASE}} or constant relaxation, buffers are not required and one directly updates $\ul{\bm{v}}_{\,*}$ and $\ul{\bm{w}}_{\,*}$.

By comparison with the constant splittings we see the following: Jacobi WR requires no relaxation during time-integration, see \mbox{\eqref{EQ JAC WR}}, since there is no dependency on the new iterate. With GS WR, see \mbox{\eqref{EQ GS WR}}, one only requires relaxation of the incoming data for the process that is not ahead. Thus, we only update the interpolant during time-integration, if the other process is ahead. We determine if a process is ahead, based on additionally communicated \textit{update indicators}. For each data-point, these indicate if it has been remotely updated. We explain the method to determine if the other process is ahead later on.

The principle data-structures involved are shown in Figure \ref{FIG AWR BUFFER}. We denote the time-grids by $\mathcal{T}_v = \{t^{(v)}_n\}_n$ and $\mathcal{T}_w = \{t^{(w)}_n\}_n$, with $\mathcal{T} = \mathcal{T}_v \cup \mathcal{T}_w$.

The local shape and thus resulting relaxation is determined by a single process and thus defined on the time-points of a single time-grid. To ensure consistent relaxation for both interpolants, we define them on the shared time-grid $\mathcal{T}$, rather than the time-grid of the respective other process, as in Algorithm \mbox{\ref{ALG WR NEW BASE}}. The base interpolant data is determined by the time-grid of the respective other process, which we interpolate to the shared grid before relaxation.

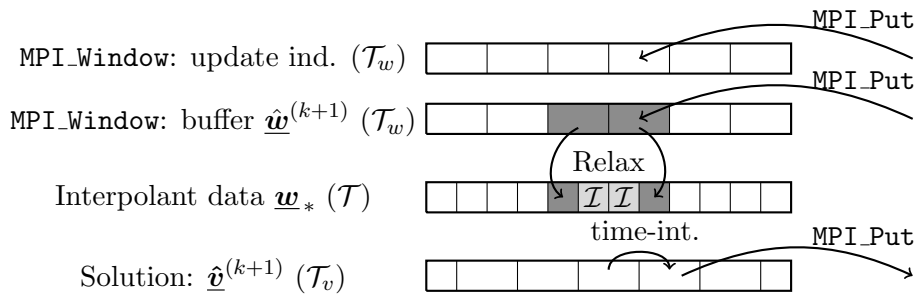
\begin{figure}[ht!]
\begin{center}
\begin{tikzpicture}[scale =0.8]
\tikzmath{\boxlen = 6;}
\tikzmath{\buffery = 0;}
\tikzmath{\indicatory = 1;}
\tikzmath{\interpy = -1.3;}
\tikzmath{\soly = -2.6;}
%
%
%
\draw[thick] (0, \indicatory) -- (\boxlen, \indicatory) -- (\boxlen, \indicatory + 0.5) -- (0, \indicatory + 0.5) -- (0, \indicatory); 
\node[align = right] at (-3.5, \indicatory + 0.25) {\texttt{MPI\_Window}: update ind. ($\mathcal{T}_w$)};
\foreach \x in {1,2,3,4,5}
    \draw (\x, \indicatory) -- (\x, \indicatory +0.5);
\draw[->, thick] (\boxlen + 2, \indicatory+0.25) [out=155, in=25] to (\boxlen -2.5, \indicatory+0.25);
\node at (\boxlen + 1.2, \indicatory + 0.9){\texttt{MPI\_Put}};
%
%
%
\fill[fill=black!45] (3,\buffery) rectangle (4,\buffery+0.5);
\fill[fill=black!45] (2,\buffery) rectangle (3,\buffery+0.5);
\draw[thick] (0, \buffery) -- (\boxlen, \buffery) -- (\boxlen, \buffery + 0.5) -- (0, \buffery + 0.5) -- (0, \buffery); 
\node[align = right] at (-3.5, \buffery + 0.25) {\texttt{MPI\_Window}: buffer $\hat{\ul{\bm{w}}}^{(k+1)}$ ($\mathcal{T}_w$)};

\foreach \x in {1,2,3,4,5}
    \draw (\x, \buffery) -- (\x, \buffery +0.5);
\draw[->, thick] (\boxlen + 2, \buffery+0.25) [out=155, in=25] to (\boxlen -2.5, \buffery+0.25);
\node at (\boxlen + 1.2, \buffery + 0.9){\texttt{MPI\_Put}};
%
%
%
\fill[fill=black!45] (3.5,\interpy) rectangle (4,\interpy+0.5);
\fill[fill=black!45] (2,\interpy) rectangle (2.5,\interpy+0.5);
\fill[fill=black!15] (2.5,\interpy) rectangle (3.5,\interpy+0.5);
\node at (2.75, \interpy+0.25) {$\mathcal{I}$};
\node at (3.25, \interpy+0.25) {$\mathcal{I}$};
\draw[thick] (0, \interpy) -- (\boxlen, \interpy) -- (\boxlen, \interpy + 0.5) -- (0, \interpy + 0.5) -- (0, \interpy); 
\node[align = right] at (-3.5, \interpy + 0.25) {Interpolant data $\ul{\bm{w}}_{\,*}$ ($\mathcal{T}$)};
\foreach \x in {1,2,3,4,5,0.5, 1.5, 2.5, 3.5, 4.5, 5.5}
    \draw (\x, \interpy) -- (\x, \interpy +0.5);
\draw[->, thick] (2.5, \buffery+0.1) [out=205, in=145] to (2.25, \interpy+0.25);
\draw[->, thick] (3.5, \buffery+0.1) [out=-25, in=45] to (3.75, \interpy+0.25);
\node at (\boxlen/2, \buffery - 0.45){Relax};
\draw[thick] (0, \soly) -- (\boxlen, \soly) -- (\boxlen, \soly + 0.5) -- (0, \soly + 0.5) -- (0, \soly); 
\node[align = right] at (-3.5, \soly + 0.25) {Solution: $\ul{\hat{\bm{v}}}^{(k+1)}$ ($\mathcal{T}_v$)};
\foreach \x in {0.5, 1.5, 2.5, 3.5, 4.5, 5.5}
    \draw (\x, \soly) -- (\x, \soly +0.5);
\draw[->, thick] (3, \soly+0.35) [out=90, in=90] to (4, \soly+0.35);
\node at (3.6, \soly+1) {time-int.};
\draw[->, thick] (4.2, \soly+0.25) [out=25, in=155] to (\boxlen + 2, \soly+0.25);
\node at (\boxlen + 1.2, \soly + 0.9){\texttt{MPI\_Put}};
\end{tikzpicture}
\end{center}
\caption{Sketch of the data-structures involved for variable relaxation from the perspective of the process solving \eqref{EQ WR DISCRETE 1}. Iteration indicators and the buffer are \texttt{MPI\_Window} objects to receive asynchronous updates via \texttt{MPI\_Put} from the other process. Their size is determined by the time-grid used on the other process. 
When performing relaxation, we update those data-points of the interpolant that are part of $\mathcal{T}_w$. Any remaining points are computed via interpolation. New time-point solutions are asynchronously communicated to the buffer of the other process, also updating the corresponding update indicators.
}
\label{FIG AWR BUFFER}
\end{figure}

W.l.o.g., consider the timestep from $t_n^{(v)}$ to $t_{n+1}^{(v)}$. To determine if the other process is ahead, we first need to find the \textit{smallest enclosing interval} $[t_-^{(w)}, t_+^{(w)}]$ with $t_-^{(w)}$, $t_+^{(w)}$ $\in \mathcal{T}_w$, such that all evaluations of the interpolant during this timestep depend on discrete data points $\bm{w}^{(k)}_n$ or $\bm{w}^{(k+1)}_n$ corresponding to the time-points $t_-^{(w)}, \ldots, t_+^{(w)}$. We determine if the other process is ahead by checking the update indicator at $t_+^{(w)}$. An example for determining the smallest enclosing interval, on non-matching time-grids, is shown in Figure \mbox{\ref{FIG AWR ENCL INTV}}.
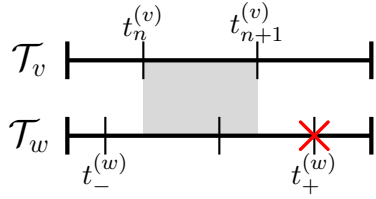
\begin{figure}[h!]
\begin{center}
\begin{tikzpicture}
\fill[fill=black!15] (2,0) rectangle (3.5,1);
\draw [ultra thick](1, 1) -- (5, 1); \node at (0.5, 1) {\Large $\mathcal{T}_v$};
\draw [ultra thick] (1, 1.25) -- (1, 0.75); 
\draw [ultra thick] (5, 1.25) -- (5, 0.75);
\draw [thick] (2, 1.25) -- (2, 0.75); \node at (2, 1.5) {$t_n^{(v)}$};
\draw [thick] (3.5, 1.25) -- (3.5, 0.75); \node at (3.5, 1.5) {$t_{n+1}^{(v)}$};

\draw [ultra thick](1, 0) -- (5, 0); \node at (0.5, 0) {\Large $\mathcal{T}_w$};
\draw [ultra thick] (1, -0.25) -- (1, 0.25); 
\draw [ultra thick] (5, -0.25) -- (5, 0.25);
\draw [thick] (1.5, 0.25) -- (1.5, -0.25); \node at (1.5, -0.5) {$t_-^{(w)}$};
\draw [thick] (3, 0.25) -- (3, -0.25);
\draw [thick] (4.25, 0.25) -- (4.25, -0.25); \node at (4.25, -0.5) {$t_+^{(w)}$};
\node at (4.25, 0) {\huge \color{red} \boldsymbol{$\times$}};
\end{tikzpicture}
\end{center}
\caption{
Consider linear interpolation and a Runge-Kutta scheme with $c_i \in [0, 1]$, then $t_-^{(w)} = \max(\{ t \in \mathcal{T}_w | t \leq t_n^{(v)}\})$ and $t_+^{(w)} = \min(\{ t \in \mathcal{T}_w | t \geq t_{n+1}^{(v)}\})$. The figure visualizes the smallest enclosing interval on non-matching $\mathcal{T}_v$ and $\mathcal{T}_w$. The cross marks the time-point for which we need to check the update indicator, to determine if the other process is ahead.}
\label{FIG AWR ENCL INTV}
\end{figure}

We use \textit{markings} for all discrete timepoints $t \in \mathcal{T}$, tracking relaxation type and if relaxation has been performed. All timepoints are unmarked at the beginning of each iteration.

In a given timestep, we mark all unmarked timepoints within the smallest enclosing interval for appropriate GS relaxation, if the other process is ahead, and for Jacobi relaxation otherwise. Additionally, if the other process is ahead, we perform relaxation on all not previously relaxed data-points for $t \in \mathcal{T}$: $t_-^{(w)} \leq t \leq t_+^{(w)}$, according to their respective markings. Finally, the actual timestep is computed. This procedure is visualized in Algorithm \ref{ALG NEW TIMESTEP}.
\begin{algfloat}[h!]
\begin{center}
\begin{tikzpicture}[node distance=1.25cm,scale=0.8, every node/.style={scale=0.8}]
\node (timestep) [box] {timestep $n$};
\node (sync) [box, below of=timestep] {Sync. buffer + update indicators};
\node (enclosing) [box, below of=sync] {Find enclosing interval};
\node (ahead) [box, below of=enclosing] {other process ahead?};
\node (markGS) [box, below of=ahead, yshift = -0.25cm] {mark points (GS)};
\node (markJAC) [box, left of=markGS, xshift = -3cm] {mark points (Jacobi)};
\node (relax) [box, below of=markGS] {relax data-points};
\node (compute) [box, below of=relax] {compute step};
\node (send) [box, right of=ahead, xshift=2.5cm] {Send data};

\draw [->] (timestep) -- (sync);
\draw [->] (sync) -- (enclosing);
\draw [->] (enclosing) -- (ahead);

\draw [->] (ahead) -- node[anchor = west] {Yes} (markGS);
\draw [->] (markGS) -- (relax);
\draw [->] (relax) -- (compute);

\draw [->] (ahead) -| node[anchor = east] {No} (markJAC);
\draw [->] (markJAC) |- (compute);

\draw [->] (compute) -| (send);
\draw [->] (send) |- node [pos = 0.7, anchor=south] {$n = n + 1$} (timestep);
\end{tikzpicture}
\caption{Timestepping procedure for variable relaxation algorithm with two processes. The synchronization of buffer and update indicators is performed using \texttt{MPI\_Win\_sync}.}
\label{ALG NEW TIMESTEP}
\end{center}
\end{algfloat}
After time-integration, the processes exchange information on which $t \in \mathcal{T}$ received GS relaxation and perform according relaxation on all non-relaxed points (over-ruling any markings for Jacobi relaxation). Afterwards, any remaining Jacobi relaxation is performed.
Lastly, the solutions corresponding to $\bm{v}^{(k+1)}(T_f)$ and $\bm{w}^{(k+1)}(T_f)$ are exchanged to facilitate a consistent termination check, c.f. \eqref{EQ WR TERMINATION CRIT}, on both processes.
%
\section{Numerical results}
%
We consider two conjugate heat transfer examples. The first one is two heterogeneous coupled linear heat equations, which is conform with the linear convergence theory presented in Section \mbox{\ref{SEC AWR CONV}}. We use it to demonstrate convergence of our new method and for a performance comparison with Jacobi and GS WR.

The second example is a gas quenching test case, inspired by \mbox{\cite{birken2012numerical,Birken2010}}, simulating cooling of a hot steel plate with pressurized air. We model the air using the compressible Euler equations and the steel plate via the nonlinear heat equation. We use a partitioned coupling of different spatial discretizations implemented using the packages \texttt{DUNE} \mbox{\cite{Bastian2021}} and \texttt{FEniCS} \mbox{\cite{logg2012_FENICS}}. This example is not conform with the assumptions on Section \mbox{\ref{SEC AWR CONV}}, since it is nonlinear and we use non-matching time-grids. Yet, we demonstrate that our new method is convergent. 
%
\subsection{Coupled heat equations}\label{SEC HEAT}
%
The model equations are
\begin{align}\label{EQ AWR HEAT D}
\begin{split}
\alpha_1 \partial_t u_1(t, \bm{x}) - \lambda_1 \Delta u_1(t, \bm{x}) = 0, 
& \quad (t,\bm{x}) \in (0, T_f] \times \Omega_1,\\
u_1(t, \bm{x}) = 0, 
&\quad (t,\bm{x}) \in [0, T_f] \times \Omega_1\setminus\Gamma,\\
u_1(t, \bm{x}) = u_\Gamma(t, \bm{x}),
&\quad (t,\bm{x}) \in [0, T_f] \times \Gamma,\\
u_1(0, \bm{x}) = u_0(\bm{x}), 
& \quad\bm{x} \in \Omega_1,
\end{split}
\end{align}
and 
\begin{align}\label{EQ AWR HEAT N}
\begin{split}
\alpha_2 \partial_t u_2(t, \bm{x}) - \lambda_2 \Delta u_2(t, \bm{x}) = 0, 
& \quad (t,\bm{x}) \in (0, T_f] \times \Omega_2,\\
u_2(t, \bm{x}) = 0, 
&\quad (t,\bm{x}) \in [0, T_f] \times \Omega_2\setminus\Gamma,\\
\lambda_2 \nabla u_2(t, \bm{x}) \cdot \bm{n}_2 = - \lambda_1 \nabla u_1 (t, \bm{x}) \cdot \bm{n}_1,
&\quad (t,\bm{x}) \in [0, T_f] \times \Gamma,\\
u_2(0, \bm{x}) = u_0(\bm{x}), 
& \quad\bm{x} \in \Omega_2.
\end{split}
\end{align}
Here, $\lambda$ is the thermal conductivity and the thermal diffusivity $D$ is defined by 
\begin{align*}
D = \lambda /\alpha, \quad \mbox{with} \quad \alpha = \rho c_p,
\end{align*}
with density $\rho$ and specific heat capacity $c_p$.

The corresponding monolithic problem is a linear heat equation defined on $\Omega_1 \cup \Omega_2$ with space dependent material parameters that have a jump in $\alpha(\bm{x})$ and $\lambda(\bm{x})$ at the interface. By enforcing continuity of temperature and heat flux at the interface, the above partitioned formulation is equivalent to the monolithic problem, in a weak sense \cite[Chap.7]{Quarteroni1999}.

\begin{figure}[h!]
\begin{center}
\begin{tikzpicture}[scale = 2]
\draw [-] (1,1) -- (3,1) -- (3,2) -- (1,2) -- (1,1);
\draw [-, thick] (2,1) -- (2,2);
\draw [-] (0.9, 1) -- (1, 1); \node at (0.8, 1) {$0$};
\draw [-] (0.9, 2) -- (1, 2); \node at (0.8, 2) {$1$};
\draw [-] (1,1) -- (1, 0.9); \node at (1, 0.8) {$-1$};
\draw [-] (2,1) -- (2, 0.9); \node at (2, 0.8) {$0$};
\draw [-] (3,1) -- (3, 0.9); \node at (3, 0.8) {$1$};

\node at (1.5, 1.3) {$\Omega_1$};
\node at (2.5, 1.3) {$\Omega_2$};

\draw [->] (2,1.5) -- (2.3, 1.5); \node at (2.15, 1.59) {$\bm{n}_1$};
\draw [->] (2,1.6) -- (1.7, 1.6); \node at (1.85, 1.69) {$\bm{n}_2$};
\node at (2.1, 1.1) {$\Gamma$};
\end{tikzpicture}
\end{center}
\caption{Geometry of the coupled heat problem.}
\label{FIG FSI OMEGA}
\end{figure}
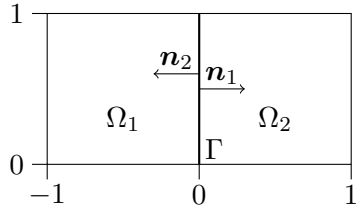
We consider $T_f = 10^4$ and the initial condition $u_0(\bm{x}) = 500\sin(\pi/2 (x_1 + 1)) \sin(\pi x_2)$. See Table \ref{TABLE MATERIALS} for the materials considered here and Figure \ref{FIG FSI OMEGA} for the geometry.
\begin{table}[ht!]
\begin{center}
\begin{tabular}{|c|c|c|}
\hline \textbf{Material} & $\alpha = \rho \cdot c_p [J/(K m^3)]$ & $\lambda [W/(m K)]$ \\
\hline Air & $1.293 \cdot 1005$
 & $0.0243$
 \\
\hline Water & $999.7 \cdot 4192.1
$ & $0.58

$ \\
\hline Steel & $7836 \cdot 443$
 & $48.9$ \\
\hline 
\end{tabular}
\caption{Material parameters.}
\label{TABLE MATERIALS}
\end{center}
\end{table}

The subproblems exchange information in the form of interface temperature $u_{\Gamma} = u_2\big|_{\Gamma}$ and the heat flux $q := \lambda_1 \nabla u_1 \cdot \bm{n}_1$.
%
\subsubsection{Discretizations}
%
We discretize \eqref{EQ AWR HEAT N} in space using linear finite elements implemented using \texttt{FEniCS} \cite{logg2012_FENICS}, on a triangulation obtained from a cartesian grid. We use the Crank-Nicolson method to discretize time. This yields
\begin{align*}
\begin{split}
&\int_{\Omega_2} \alpha_2 (u^{(2)}_{n+1} - u^{(2)}_n) \varphi + \frac{\Delta t}{2} \lambda_2 \nabla \left(u^{(2)}_{n+1} + u^{(2)}_n\right) \cdot \nabla \varphi\, \intd{\bm{x}} \\ 
& + \frac{\Delta t}{2}\int_\Gamma \left(q_{n+1} + q_n\right) \varphi\, \intd{\bm{S}} = 0
,\quad \forall \varphi \in V,
\end{split}
\end{align*}
with an appropriate finite element space $V$. The weak form of \mbox{\eqref{EQ AWR HEAT D}} is obtained by omitting the heat flux $q$ and including the Dirichlet-boundary condition at $\Gamma$.
On $\Omega_1$ we use the above weak form to compute the heat flux $q_{n+1} \approx q(t_{n+1})$, based on $u^{(1)}_{n+1}$, $u^{(1)}_n$ and $q_{n}$. We compute the initial flux $q_0$ required for the interpolant from the initial condition via
\begin{equation*}
q_0 = \lambda_1 \int_\Gamma (\nabla u_0 \cdot \bm{n}_1) \varphi\, \intd{S}.
\end{equation*}
The WR methods are implemented following the partitioned approach, treating the space-discretizations of the subsolvers as black-boxes. That is, instead of exchanging discrete interface unknowns, the subsolvers exchange the interface temperatures resp. heat fluxes corresponding to points at the interface $\Gamma$.
%
\subsubsection{Relaxation}\label{SEC NUM RELAX}
%
We use single parameter relaxation on either one or both of the discrete exchange variables $\bm{u}_{\Gamma}$ and $\bm{q}$, i.e.,
\begin{align*}
\bm{u}_{\Gamma}^{(k+1)}(t)
&=
\Theta(t)^{(k+1)} \hat{\bm{u}}_{\Gamma}^{(k+1)}
+ (1 - \Theta(t)^{(k+1)}) \bm{u}_{\Gamma}^{(k)}, \\
\bm{q}^{(k+1)}(t)
&=
\Theta(t)^{(k+1)} \hat{\bm{q}}^{(k+1)}
+ (1 - \Theta(t)^{(k+1)}) \bm{q}^{(k)},
\end{align*}
with $\Theta(t)^{(k+1)} \in \mathbb{R}$. To perform the algorithm from Section \ref{SEC AWR VAR RELAX}, we require relaxation parameters for Jacobi and GS WR.

Optimal constant relaxation on $\bm{u}_{\Gamma}$ for time-discrete GS WR in the $\Omega_1 \rightarrow \Omega_2$ order has been determined in \cite{Monge2018a,Monge2017}. This was done for 1D linear finite elements on a uniform space discretization and implicit Euler for constant and matching step-sizes. However, results from \cite{Meisrimel} show $\Theta_{opt}$ to be robust, working well in 2D and with the second order in time SDIRK2 scheme. 

Here, we use the results in \cite{Monge2018a} to determine optimal relaxation for the remaining local shapes, i.e., Jacobi and GS in the $\Omega_2 \rightarrow \Omega_1$ order. In the $\Omega_1 \rightarrow \Omega_2$ order, the iteration, including relaxation, for $\bm{u}_\Gamma$ and $\bm{q}$ is given by the following relations \cite{Monge2018a}:
\begin{equation*}
\begin{pmatrix}
\bm{u}_\Gamma \\ \bm{q}
\end{pmatrix}^{(k+1)}
= 
\begin{pmatrix}
(1 - \Theta) \bm{I} - \Theta {\bm{S}^{(2)}}^{-1}\bm{S}^{(1)} & \bm{0} \\
\bm{0} & \bm{S}^{(1)}
\end{pmatrix}
\begin{pmatrix}
\bm{u}_\Gamma \\ \bm{q}
\end{pmatrix}^{(k)}
+
\boldsymbol{\psi}^{(k)}.
\end{equation*}
Here, $\boldsymbol{\psi}^{(k)}$ are additional terms irrelevant to the iteration matrix. In the 1D case, where $\bm{S}^{(m)} \in \mathbb{R}$, the optimal choice is $\Theta_{opt} = \frac{1}{\left| 1 + {\bm{S}^{(2)}}^{-1} \bm{S}^{(1)}\right|}$, yielding a zero spectral radius for the above iteration matrix. Analytical expressions for $\bm{S}^{(m)}$ have been computed in \cite{Monge2018a}, which are
\begin{align*}
\bm{S}^{(m)}
& = 
\frac{6 \Delta t \Delta x (\alpha_m \Delta x^2 + 3 \lambda_m \Delta t) - (\alpha_m \Delta x^2 - 6 \lambda \Delta t)^2s_m}{18 \Delta t \Delta x^3},\\
s_m
& = 
\sum_{i=1}^{N} \frac{3 \Delta t \Delta x^2 \sin^2 (i \pi \Delta x)}{2 \alpha_m \Delta x^2 + 6 \lambda_m \Delta t + (\alpha_m \Delta x^2 - 6 \lambda_m \Delta t) \cos (i \pi \Delta x)}.
\end{align*}
%
Similarly, the relations for GS in the $\Omega_2 \rightarrow \Omega_1$ order, with relaxation on $\bm{q}$, are
\begin{equation*}
\begin{pmatrix}
\bm{u}_\Gamma \\ \bm{q}
\end{pmatrix}^{(k+1)}
= 
\begin{pmatrix}
\bm{0} & -{\bm{S}^{(2)}}^{-1}\\
\bm{0} & (1 - \Theta) \bm{I} -\Theta \bm{S}^{(1)}{\bm{S}^{(2)}}^{-1} & 
\end{pmatrix}
\begin{pmatrix}
\bm{u}_\Gamma \\ \bm{q}
\end{pmatrix}^{(k)}
+
\boldsymbol{\psi}^{(k)}.
\end{equation*}
Relaxation with the same $\Theta_{opt}$ as in the $\Omega_1 \rightarrow \Omega_2$ case yields a zero spectral radius in 1D. Jacobi WR with relaxation on both $\bm{u}_\Gamma$ and $\bm{q}$ using the same $\Theta$ yields
\begin{equation*}
\begin{pmatrix}
\bm{u}_\Gamma \\ \bm{q}
\end{pmatrix}^{(k+1)}
= 
\begin{pmatrix}
(1- \Theta)\bm{I} & -\Theta{\bm{S}^{(2)}}^{-1}\\
\Theta\bm{S}^{(1)} & (1- \Theta)\bm{I}
\end{pmatrix}
\begin{pmatrix}
\bm{u}_\Gamma \\ \bm{q}
\end{pmatrix}^{(k)}
+
\boldsymbol{\psi}^{(k)}.
\end{equation*}
In 1D and with $\bm{S}^{(1)} {\bm{S}^{(2)}}^{-1} > 0$, which is the case for all material combinations here considered, the spectral radius is minimal for $\Theta = 1/(\bm{S}^{(1)} {\bm{S}^{(2)}}^{-1} + 1)$. In particular, the spectral radius with optimal relaxation is
\begin{equation*}
\rho_{opt}^{\text{Jacobi}} = 
\sqrt{\frac{\bm{S}^{(1)} {\bm{S}^{(2)}}^{-1}}{\bm{S}^{(1)} {\bm{S}^{(2)}}^{-1} + 1}}.
\end{equation*}
With $\Delta x = 1/513$, $\Delta t = 5$ and $\alpha_m, \lambda_m$ for air, water and steel, see Table \ref{TABLE MATERIALS}, the spectral radii are $\rho_{\text{air-steel}}^{\text{Jacobi}} = 0.037$, $\rho_{\text{air-water}}^{\text{Jacobi}} = 0.059$ and $\rho_{\text{water-steel}}^{\text{Jacobi}} = 0.528$.
%
\subsubsection{Results}
%
All numerical experiments were run on an Intel i5-2500K 3.30 GHz CPU with Python 3.6.9, Open MPI 2.1.1, \texttt{FEniCS} 2019.2.0.dev0 \cite{logg2012_FENICS}. The code is available at \cite{Meisrimel2020}.

For GS we denote the different orderings of GS WR by "GS\_DN" and "GS\_ND". We use the result for GS in the "DN" ($\Omega_1$ first) order with $TOL_{WR} = 10^{-11}$ as the reference result. Despite the asynchronous method being non-deterministic, results for $5$ simulations showed no notable deviations in the number of iterations, we show the mean result.

First, we consider the error of the interface temperature $\bm{u}_\Gamma$, over $k$, for $TOL_{WR} = 10^{-10}$, $\Delta x = 1/513$ and $N = 200$, resulting in a comparable accuracy in space and time. We use $\bm{u}_\Gamma^{(k)}$ in the discrete interface $\mathcal{L}^2$ norm for the termination check \eqref{EQ WR TERMINATION CRIT} and error computation.

Results in Figure \ref{FIG AWR ERR OVER ITER} show convergence for all considered WR methods, numerically verifying the result of Theorem \mbox{\ref{THRM AWR DISCRETE CONV}}. The convergence rates of our new asynchronous method are in between Jacobi and GS in all test cases. Performance results in Figure \ref{FIG AWR ERR TIME FENICS FENICS} show slight performance improvements compared to the constant splitting WR methods. This is since our new method requires less than twice the number of iterations of GS WR and is parallel.

\begin{figure}[h!]
\includegraphics[width=5.3cm]{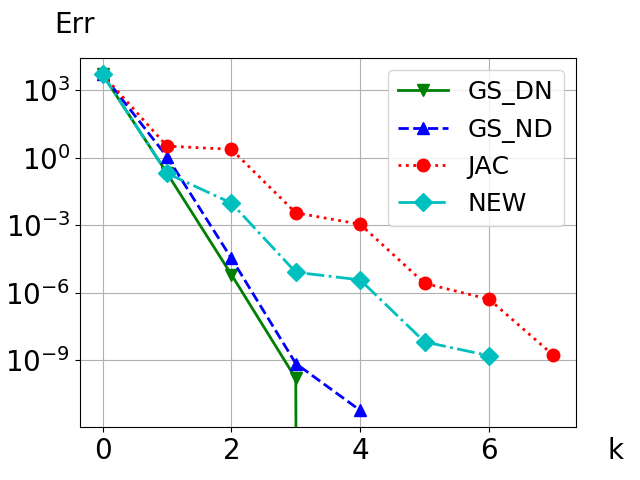}
\includegraphics[width=5.3cm]{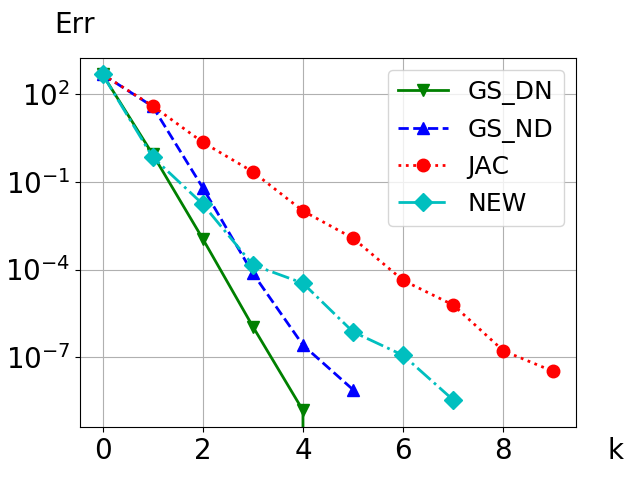}
\includegraphics[width=5.3cm]{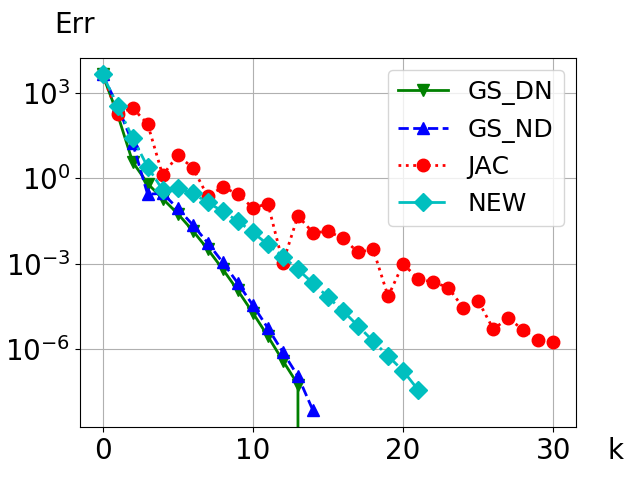}
\caption{\textbf{Left to right}: Air-steel, air-water and water-steel. $k$ marks the number of iterations until \eqref{EQ WR TERMINATION CRIT}, with $TOL_{WR} = 10^{-10}$, is met.}
\label{FIG AWR ERR OVER ITER}
\end{figure}

\begin{figure}[h!]
\includegraphics[width=5.3cm]{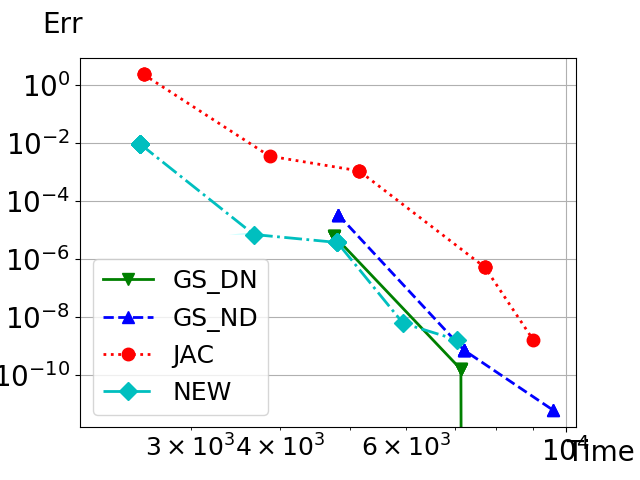}
\includegraphics[width=5.3cm]{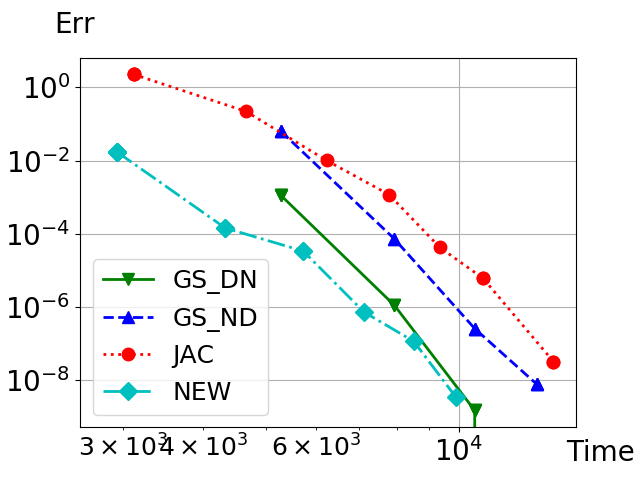}
\includegraphics[width=5.3cm]{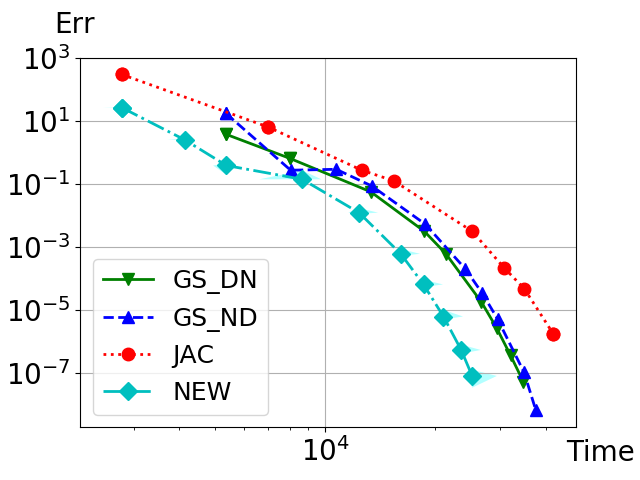}
\caption{\textbf{Left to right}: Air-steel, air-water and water-steel. We use \texttt{MPI\_Wtime} to measure the wall-clock runtime.}
\label{FIG AWR ERR TIME FENICS FENICS}
\end{figure}
%
\subsection{Gas quenching}
%
We now consider an example for the gas quenching application \cite{Yarrington1994,birken2012numerical}; cooling a hot (steel) work-piece with air. We model the fluid (air) via the compressible Euler equations
\begin{align}\label{EQ WR EULER EULER}
\begin{split}
\partial_t \rho + \nabla \cdot \rho \,\bm{v} & = 0, \\
\partial_t \rho v_i + \sum_{j=1}^2 \partial_{x_i} (\rho v_i v_j + \delta_{ij} p) & = 0, \quad i = 1,2, \\
\partial_t \rho E + \nabla \cdot (\rho H \bm{v}) & = 0.
\end{split}
\end{align}
Here $\rho$, $p$ and $v_1,\, v_2$ are density, pressure and velocities. The enthalpy is $H = E + p/\rho$ with total energy (per unit mass) $E = e + |\bm{v}|^2/2$ and specific internal energy $e$. The system is completed with the ideal gas law $p = (\gamma - 1)\rho \, e$, where $\gamma$ is the adiabatic exponent, here $\gamma = 1.4$. The temperature is $T = p/(\rho \cdot R)$ with the specific gas constant $R = 287.058$ for dry air.

We model the solid by the nonlinear heat equation 
\begin{align}\label{EQ WR EULER HEAT EQ}
\begin{split}
\alpha(u) \partial_t u(t, \bm{x}) - \nabla \cdot \left( \lambda(u) \nabla u(t, \bm{x})\right) = 0, 
& \quad (t,\bm{x}) \in (0, T_f] \times \Omega_2,\\
\nabla u(t, \bm{x}) \cdot \bm{n} = 0, 
&\quad (t,\bm{x}) \in [0, T_f] \times \partial\Omega_2\setminus\Gamma,\\
\lambda(u) \nabla u(t, \bm{x}) \cdot \bm{n}_\Gamma = q(t, \bm{x}),
&\quad (t,\bm{x}) \in [0, T_f] \times \Gamma,\\
u(0, \bm{x}) = u_0^{\text{solid}}, 
& \quad\bm{x} \in \Omega_2.
\end{split}
\end{align}
Here, we use the material parameters of 51CrV4 steel from \cite{Quint2011}, which are given by
\begin{align*}
\lambda(u)
& = 40.1 + 0.05 u - 0.0001 u^2 + 4.9e-8 u^3
,\quad
\alpha(u) 
= 7836 c_p(u)
\\
c_p(u) 
&= -10 \text{ln}\left(\left( \text{e}^{c_{p1}(u)/10} + \text{e}^{c_{p2}(u)/10}\right)/2\right),
\\
c_{p1}(u)
&= 34.2 \text{e}^{0.0026 u} + 421.15 
,\quad
c_{p2}(u)
= 956.5 \text{e}^{-0.012(u - 900)} + 0.45u.
\end{align*}
The problem geometry, inspired by \cite{Birken2010}, is shown in Figure \ref{FIG HEAT EULER}.
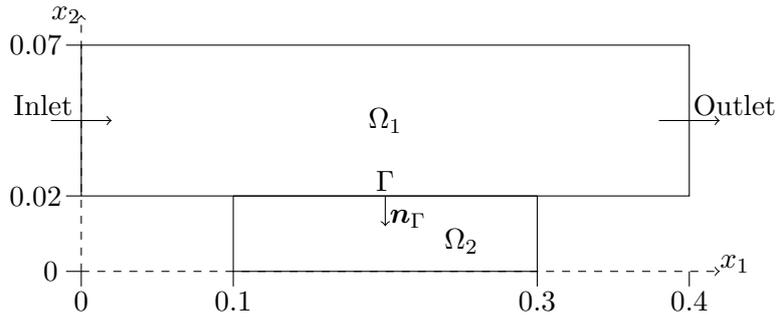
\begin{figure}[h!]
\begin{center}
\begin{tikzpicture}[scale = 2]
\draw [-] (1,1) -- (5,1) -- (5,2) -- (1,2) -- (1,1);
\node at (3, 1.5) {$\Omega_1$};

\draw [-] (2,1) -- (4,1) -- (4,0.5) -- (2,0.5) -- (2,1);
\node at (3.5, 0.7) {$\Omega_2$};

\node at (3, 1.1) {$\Gamma$};
\draw[->] (3, 1) -- (3, 0.8); \node at (3.15, 0.85) {$\bm{n}_{\Gamma}$};

\draw[->] (0.8, 1.5) -- (1.2, 1.5); \node at (0.75, 1.6) {Inlet};
\draw[->] (4.8, 1.5) -- (5.2, 1.5); \node at (5.3, 1.6) {Outlet};
\draw[dashed, ->] (1,0.5) -- (5.2, 0.5); \node at (5.3, 0.55) {$x_1$};
\draw[dashed, ->] (1,0.5) -- (1, 2.2); \node at (0.9, 2.2) {$x_2$};
\draw[-] (1,0.5) -- (1, 0.4); \node at (1, 0.3) {$0$};
\draw[-] (2,0.5) -- (2, 0.4); \node at (2, 0.3) {$0.1$};
\draw[-] (4,0.5) -- (4, 0.4); \node at (4, 0.3) {$0.3$};
\draw[-] (5,0.5) -- (5, 0.4); \node at (5, 0.3) {$0.4$};
\draw[-] (1,0.5) -- (0.9, 0.5); \node at (0.8, 0.5) {$0$};
\draw[-] (1,1) -- (0.9, 1); \node at (0.7, 1) {$0.02$};
\draw[-] (1,2) -- (0.9, 2); \node at (0.7, 2) {$0.07$};
\end{tikzpicture}
\end{center}
\caption{Geometry, not to scale, for the gas quenching test case.}
\label{FIG HEAT EULER}
\end{figure}

In $\Omega_1$ we define $\hat{u}^0_{fluid} := (\rho_0, \rho_0 v_0, 0, E_0)$, with $\rho_0 = 1.225$, $T_0 = 273.15$ and $v_0 = 0.8 \text{Ma}$. We use $\hat{u}^0_{fluid}$ for the inlet boundary and the farfield boundary at the top, on the outlet we extrapolate interior values to obtain a zero flux. On the bottom we employ a slip boundary condition. Additionally, at the interface $\Gamma$ we use $T_\Gamma$, on the left we set the wall temperature to $T_0$ and on the right we use $\nabla T \cdot \bm{n}_\Gamma = 0$.

On $\Omega_1$ we discretize \eqref{EQ WR EULER EULER} in space using a $1$st order finite volume discretization, implemented in \texttt{DUNE} \cite{Bastian2021,Dedner2020}. The is grid shown in Figure \ref{FIG HEAT EULER GRID}. We discretize \eqref{EQ WR EULER HEAT EQ} in space using linear finite elements, implemented using \texttt{FEniCS} \cite{logg2012_FENICS}, on a triangulation obtained from splitting the squares of a cartesian grid with $\Delta x = 0.00125$, matching $\Omega_1$ at the interface. 
\begin{figure}[h!]
\begin{center}
\includegraphics[width = 16cm]{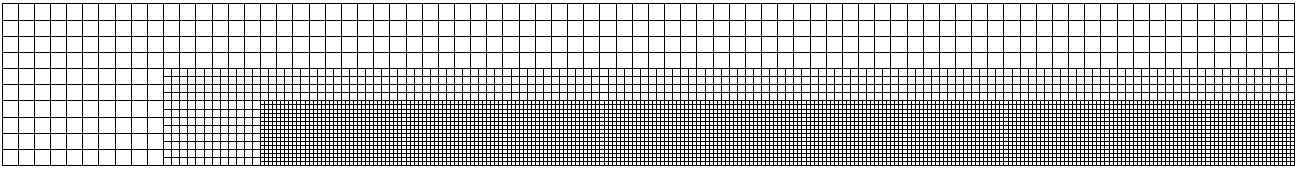}
\end{center}
\caption{Grid used in gas quenching test case. Starting from a cartesian grid with $\Delta x = 0.005$ we perform two steps of local grid refinement on all cells whose centres $(x_1, x_2)$ fulfill the following criteria: First step: $x_1 > 0.05$, $x_2 < 0.03$ and second step: $x_1 > 0.08$, $x_2 < 0.02$. The resulting grid consists of $5132$ cells.}
\label{FIG HEAT EULER GRID}
\end{figure}

For the time-discretizazions, we use the Crank-Nicolson method with $\Delta t = 10^{-2}$, c.f. Section \ref{SEC HEAT}, for the solid. In the fluid we use the SDIRK2 method with $\Delta t = 2 \cdot 10^{-4}$, solving the nonlinear systems using a Jacobian-free Newton-Krylov method with ILU preconditioning. 

We obtain the initial condition $u^0_{fluid}$ on $\Omega_1$ for the coupled simulation as follows: We compute a stationary solution with $T_\Gamma = 900$, starting with initial conditions $\hat{u}^0_{fluid}$ and simulating until $T_f = 0.01$ with a step-size of $\Delta t = 10^{-5}$. On $\Omega_2$ we use the constant initial value $u_0^{\text{solid}} = 900$.
%

The problems are coupled using a Dirichlet-Neumann approach. That is, on $\Omega_1$ we compute the heat-flux $q$ as the interface boundary condition on $\Omega_2$ and on $\Omega_1$, we use the interface temperature $T_\Gamma = u \big|_\Gamma$, computed in $\Omega_2$.

We compute the discrete heat-flux by computing 
\begin{equation}
q = \kappa (T^{\,*}) \nabla T^{\,*} \cdot \bm{n}_\Gamma,
\end{equation}
at the interface, where $T^{\,*}$ is the temperature on $\Omega_1$, linearly reconstructed. The heat conductivity $\kappa$ in the fluid is given via the Sutherland law by
\begin{equation*}
\frac{\kappa(T\,)}{\kappa_{ref}} = \left(\frac{T}{T_{ref}}\right)^{3/2} \frac{T_{ref} + S_k}{T + S_k},
\end{equation*}
with $\kappa_{ref} = 0.0241$, $T_{ref} = 273$ and $S_k = 194$.

For the coupled run, we successively perform WR until $T_f = 5$, on time-windows of length $0.1$. We use the same relaxation as Section \ref{SEC HEAT}, with relaxation parameters for the air-steel material combination, see Table \ref{TABLE MATERIALS}, $\Delta t = 10^{-2}$ and $\Delta x = 0.00125$, which is the cell-size at the interface. 

Here, the computational load of the subsolvers is not well balanced, since the fluid problem has more unknowns and requires smaller step-sizes. As a consequence, our new method is almost identical to GS WR. The load-balancing problem can be resolved by allocating more processors in the space discretiazion of the fluid problem. This is subject to future work and required for a sensible performance comparison with Jacobi and GS WR.

Figure \ref{FIG EULER TIP TEMP} shows the WR updates for the first $3$ time-windows, which shows the iteration converges rapidly, as expected \mbox{\cite{Monge2017}}. For larger tolerances, chosen in accordance with the errors in space and time, one requires at most $3-4$ iterations per time-window.

Figure \ref{FIG EULER TIP TEMP} also shows the temperature at the left top tip of the steel plate over time, showing a steady cooling effect. Figures \ref{FIG EULER RES FLOW} visualizes the temperature at $T_f = 5$. 

\begin{figure}[h!]
\begin{center}
\includegraphics[width = 7.5cm]{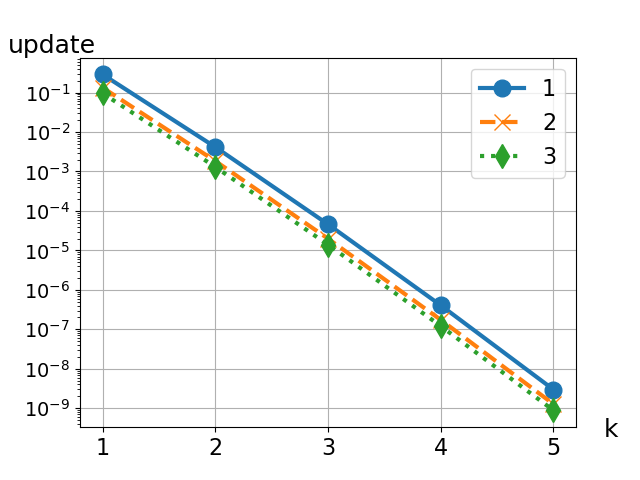}
\includegraphics[width = 7.5cm]{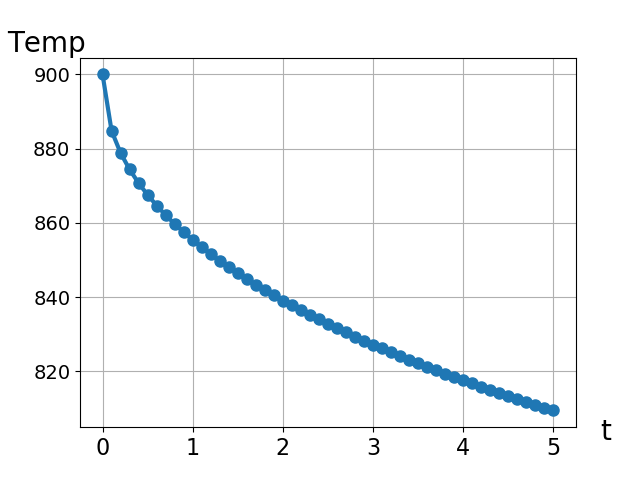}
\end{center}
\caption{\textbf{Left}: Updates of the interface temperatures in the $2$-norm in the first $3$ time-windows. \textbf{Right}: Temperature at the top left tip of the steel plate, $(x_1, x_2) = (0.1, 0.02)$, c.f., Figure \ref{FIG HEAT EULER}, over time.}
\label{FIG EULER TIP TEMP}
\end{figure}

\begin{figure}[h!]
\begin{center}
\includegraphics[width = 15cm]{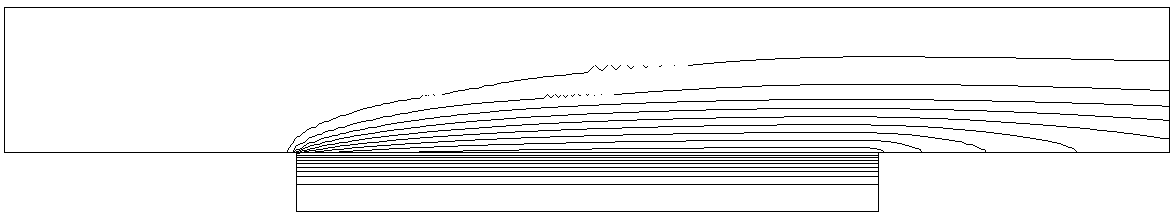}
\end{center}
\caption{Isolines of the temperatures at $T_f = 5$. Here, the isolines in $\Omega_1$ and $\Omega_2$ are distinct.}
\label{FIG EULER RES FLOW}
\end{figure}
%
\section{Summary and conclusions}
%
In this paper we presented a novel parallel WR method utilizing asynchronous communication during time-integration. Here, \texttt{MPI} One-sided communication is instrumental in the implementation of this method. The analytical description and convergence results of our new method extend existing linear WR theory by including splittings variable in time and iteration. We present an algorithm for performing optimal variable relaxation for two coupled problems.

Our new method is unconditionally parallel due to the use of asynchronous communication. If used for coupled problems with poor load-balancing, our method is equivalent to classical Gauss-Seidel WR. 

Numerical results demonstrate convergence of our method, with a convergence rate faster than Jacobi WR. Additionally, a performance comparison with two subsolvers with an approximately equal computational workload, show performance improvements for our new method, in comparison with Gauss-Seidel and Jacobi WR. 

WR methods enable the coupling of subsolvers in a partitioned manner. For our numerical experiments we implemented our PDE subsolvers using the open source packages \texttt{DUNE} \cite{Bastian2021} and \texttt{FEniCS} \cite{logg2012_FENICS}, and developed suitable adapters to facilitate this coupling. This coupling is shown to work well in a conjugate heat transfer test case. There, we couple the compressible Euler equations with a nonlinear heat equation, using finite volume resp. linear finite element discretizations in space.
%
\subsubsection*{Acknowledgments}
We'd like to thank Joachim Hein for help with \texttt{MPI}, Robert Kl\"{o}fkorn for help in implementing the \texttt{DUNE} subsolver and Benjamin Rodenberg for collaboration in developing the coupling to \texttt{FEniCS}.
\bibliographystyle{IEEEtranS}
\bibliography{IEEEabrv,literature.bib}
\end{document}